\documentclass[hidelinks,onefignum,onetabnum]{siamart250211}

\usepackage{amsmath,amssymb}
\usepackage{graphicx}
\usepackage{subcaption}
\usepackage{todonotes}
\usepackage{algorithmicx}
\usepackage{algpseudocode}
\usepackage{mathtools}
\usepackage{mathrsfs}
\algdef{SE}[DOWHILE]{Do}{doWhile}{\algorithmicdo}[1]{\algorithmicwhile\ #1}%
\def\reg#1{\II}

\newcommand{\newthm}[2]{
  \theoremstyle{plain}
  \theoremheaderfont{\normalfont\sc}
  \theorembodyfont{\normalfont\itshape}
  \theoremseparator{.}
  \theoremsymbol{}
  \newtheorem{#1}[theorem]{#2}
}

\newthm{redlemma}{Lemma}
\newsiamremark{remarkN}{Remark}

\def\R{{\mathbb R}}

\def\N{{\mathbb N}}

\def\MM{{\mathcal M}}

\def\PP{{\mathcal P}}

\def\II{{\mathcal I}}

\def\SS{{\mathcal S}}

\def\TT{{\mathcal T}}

\def\UU{\mathcal U}

\def\XX{{\mathcal X}}
\def\YY{{\mathcal Y}}

\def\tend{{t_{\rm end}}}

\def\bf{{\boldsymbol{f}}}

\def\norm#1#2{\|#1\|_{#2}}

\def\set#1#2{\big\{#1\,:\,#2\big\}}

\def\v{\mathbf{v}}

\def\normL2#1#2{\|#1\|_{L^2(#2)}}

\newcommand{\dual}[3][]{#1\langle#2\,,\,#3#1\rangle}

\def\T{\mathbb T}

\headers{Optimal Time-Adaptivity}{M. Feischl, F. Henr\'iquez, D. Niederkofler}
\title{Optimal Time-Adaptivity for Parabolic Problems\thanks{Submitted to the editors on \today.
\funding{Funded by the Deutsche Forschungsgemeinschaft (DFG, German Research Foundation) -- Project-ID 258734477 -- SFB 1173, the Austrian Science Fund (FWF)
under the special research program Taming complexity in PDE systems (grant SFB F65) as well as project I6667-N. Funding was received also from the European Research Council (ERC) under the
European Union’s Horizon 2020
research and innovation programme (Grant agreement No. 101125225).}}}

\author{Michael Feischl, Fernando Henr\'iquez, and David Niederkofler\thanks{Institute for Analysis and Scientific Computing, Vienna University of Technology, Wiedner Hauptstra{\ss}e 8-10, A-1040 Wien, Austria. E-mail addresses: {\tt\{michael.feischl,fernando.henriquez,david.niederkofler\}@asc.tuwien.ac.at}}.}

\begin{document}
\maketitle
\begin{abstract}
 Since the first optimality proofs for adaptive mesh refinement algorithms in the early 2000s, the theory of optimal mesh refinement for PDEs was inherently limited to stationary problems. The reason for this is that time-dependent problems usually do not exhibit the necessary coercive structure that is used in optimality proofs to show a certain quasi-orthogonality, which is crucial for the theory. Recently, by using a new equivalence between quasi-orthogonality and inf-sup stability of the underlying problem, it was shown that an adaptive Crank-Nicolson scheme for the heat equation is optimal under a severe step size restriction. In this work, we use this new approach towards quasi-orthogonality together with  Radau IIA methods of any order larger than one to obtain the first adaptive time stepping method for non-stationary PDEs that is provably rate optimal with respect to number of time steps vs. approximation error. 
\end{abstract}

\section{Introduction}
The theory of optimal adaptive mesh refinement originated from breakthrough results by Binev, Dahmen, DeVore~\cite{bdd}, Stevenson~\cite{stevenson07}, and Cascon, Kreuzer, Nochetto, Siebert~\cite{ckns}, who showed that a standard adaptive loop of the form
\begin{align*}
    \fbox{Solve}\longrightarrow\fbox{Estimate}\longrightarrow\fbox{Mark}\longrightarrow\fbox{Refine}
\end{align*}
produces optimal convergence rates for the Poisson problem. The new ideas inspired a flurry of research in this area, extending the results to many other model problems, see e.g.,~\cite{ks,cn} for conforming methods,~\cite{BeMao10,ncstokes3}    
for nonconforming methods,~\cite{LCMHJX,CR2012} for mixed formulations, and~\cite{fkmp,gantumur} for boundary integral equations.
For a comprehensive overview, we refer to~\cite{carstensen_2014,actaadaptive}. Roughly speaking, the strategy to show that an adaptive algorithm is optimal in terms of convergence rate is the following: 
\begin{enumerate}
    \item[(A)] Derive an error estimator $\eta$ that is an upper bound for the approximation error.
    \item[(B)] Confirm that one step of the adaptive algorithm results in a perturbed reduction of the error estimator, i.e., $\eta_{\ell+1}\leq q \eta_\ell + {\rm pert}_\ell$
    \item[(C)] Check that the perturbations ${\rm pert}_\ell$ are summable in a certain way. This usually follows from orthogonality properties (quasi-orthogonality).
    \item[$\implies$] Use (A)--(C) to show optimality of the algorithm.
\end{enumerate}
For the Poisson equation with standard residual based error estimator and nested ansatz spaces for example, the perturbations are summable due to the Galerkin orthogonality of the Galerkin solutions together with the symmetry of the problem. This crucial part of the argument does not transfer to time-dependent problems, which, at least in their standard form, usually lack symmetry and coercivity. Even indefinite or nonsymmetric stationary problems (such as the Stokes problem) posed a big hurdle for the theory of optimal mesh refinement.

There are two possible ways out of this: One option is to find a non-standard discretization of the time-dependent problem that has symmetry and coercivity, see, e.g.,~\cite{diening2025} where the authors find a symmetric reformulation of the heat equation in non-standard Sobolev spaces, or~\cite{FUHRER202127,GaSte24} where a least squares reformulation is used for adaptivity.  These non-standard discretizations avoid the lack of coercivity but come with other difficulties that so far have prevented optimality proofs. 
We also mention~\cite{adaptiveG,adaptiveDG}, which develop adaptive time stepping algorithms for parabolic problems that are guaranteed to reach the final time and obey a given error tolerance.

On the other hand, adaptive wavelet discretizations~\cite{wavelet_schwab} are known to be optimal for quite some time now but come with challenging implementation and other restrictions. See also~\cite{halfwavelet} for a hybrid, wavelet-in-time approach.

The second option is to prove optimality without relying on coercivity and symmetry.
This problem was tackled recently in~\cite{infsup}. Roughly speaking, the result shows that if the underlying discrete method is uniformly inf-sup stable, the quasi-orthogonality~(C) is automatically true. This is proved by exploiting a connection between quasi-orthogonality and the stability of the LU-factorization of matrices and removes a major hurdle for optimality proofs of non-symmetric problems. In~\cite{infsup} optimality of the adaptive algorithm is shown for the Stokes problem with Taylor-Hood elements, for a transmission problem with finite-element/boundary-element coupling, and, most important for the present work, for adaptive time stepping with the Crank-Nicolson method for the heat equation. 
The last result, however, is only true under a severe and unrealistic step size restriction of the form $\tau\lesssim h^2$, where $\tau$ is the time step size and $h$ is the size of the spatial elements. 

\bigskip

The goal of this work is to lift this restriction and to propose the first provably optimal adaptive algorithm for a time dependent problem. 
The natural idea is to apply the results from~\cite{infsup} to a better time stepping method that does not require a CFL condition.
The main difficulty of this approach, however, is that the results in~\cite{infsup} require a time stepping scheme that can be equivalently written as a Petrov-Galerkin method with certain ansatz and test spaces $\XX_\TT$, $\YY_\TT$ corresponding to a sequence of time steps $\TT$, i.e., the time stepping approximation $u_\TT$ must be the unique solution of
\begin{align*}
    a(u_\TT, v) = f(v)\quad\text{ for all }v\in \YY_\TT.
\end{align*}
Moreover, the spaces must be nested, i.e., $\XX_\TT\subseteq \XX_{\TT'}$ and $\YY_\TT\subseteq \YY_{\TT'}$ if $\TT'$ is a finer sequence than $\TT$ and the method must be inf-sup stable in the sense
\begin{align*}
    \inf_{u\in \XX_\TT}\sup_{v\in \YY_\TT}\frac{ a(u,v)}{\norm{u}{\XX}\norm{v}{\YY}}\geq c_0>0,
\end{align*}
with some constant $c_0>0$ which is independent of $\TT$. 

\bigskip

These conditions prevent a straightforward proof for commonly used schemes such as the implicit Euler method or the Crank-Nicolson method.
While the Crank-Nicolson method can be written as a Petrov-Galerkin scheme with $\XX_\TT$ being piecewise linear and $\YY_\TT$ piecewise constant functions in time, the inf-sup stability holds only under the step size restriction discussed above. On the other hand, the implicit Euler method satisfies the inf-sup condition, however, we could only find conforming Petrov-Galerkin interpretations with non-nested test spaces.

\bigskip

To overcome this problem, we first considered a time stepping method that can be seen as a hybrid of Crank-Nicolson and implicit Euler.
We arrive at this method by forcing the residual of the equation to have zero integral mean and to be pointwise zero at the endpoint of each time interval. It turned out that such a method is equivalent to the third order Radau IIA method if the data is elementwise quadratic (see, e.g.,~\cite{hairer}), which is a collocation method that evaluates at the points $1/3$ and $1$ relative to each time interval.
Consequently, we find conforming Petrov-Galerkin interpretations of the Radau IIA methods of any order larger than one with nested ansatz and test spaces by realizing the point evaluation with its Riesz representer.
These higher-order Radau IIA methods are known to be inf-sup stable for slightly different test spaces since~\cite{andreev_rk}. (We refer to~\cite{infsupsmears} for a treatment of time step wise inf-sup stability of a dG-reformulation of the Radau IIA method (among others))
We derive a residual based error estimator that satisfies the properties~(A)--(B) from above and use~\cite{infsup} to show~(C) and thus optimality of the corresponding adaptive algorithm.

To summarize, the main contributions of the present work are:
\begin{itemize}
    \item A new interpretation of the Radau IIA methods of order $\geq 2$ as Petrov-Galerkin schemes with nested ansatz and test spaces.
    \item The first proof of optimal convergence rates for a time-adaptive method for parabolic problems without CFL condition.
\end{itemize}

\section{Model Problem \& Discretization}\label{sec:model_problem}
We consider the abstract heat equation on the time interval $[0,\tend]$ for the finite dimensional Gelfand triple $V\subseteq H\subseteq V^\star$ (note that $V$, $H$, and $V^\star$ contain the same elements), i.e., for an elliptic and self-adjoint operator $A\colon V\to V^\star$, $u_0\in H$, and $f\in H^1(0,\tend;V^\star)$, we solve
\begin{align}\label{eq:parabolic_problem}
(\partial_t+A)u &= f\quad\text{in } [0,\tend]\times V^\star,\\
u(0)&=u_0\quad\text{in } H.
\end{align}
Note that the natural space for the solution is $u\in \XX:= L^2(0,\tend;V)\cap H^1(0,\tend;V^\star)$
and we define the natural test space $\YY:=L^2(0,\tend;V)$. We denote the $V^\star,V$ duality brackets and the $H$-inner product by $\dual{\cdot}{\cdot}$, where its meaning will be clear from the context. We slightly abuse notation and write $Au$ to indicate the pointwise in-time application of $A$ (or $A^{-1}$), i.e. $(Au)(t):=Au(t)$. We repeat the well-known inf-sup stability result, which can be found in e.g. \cite[Theorem 5.1]{wavelet_schwab}, for completeness.
\begin{lemma}\label{lem:infsupcont}
Given $u\in\XX$, the test function  $v:=A^{-1}(\partial_t u + Au) \in \YY$ satisfies
\begin{align}\label{eq:infsup}
\begin{split}
    C\norm{(\partial_t +A)u}{L^2(0,\tend;V^\star)}^2 
&\geq  \int_0^\tend \dual{(\partial_t +A)u}{ v}\,dt\\
&\geq c \norm{u}{\XX}^2 + \norm{u(\tend)}{H}^2-\norm{u(0)}{H}^2
\end{split}
\end{align}
and $\norm{v}{\YY}\leq \sqrt{2} \norm{A^{-1}}{V^\star \to V}\norm{u}{\XX}$. The ellipticity constant $c_A$ of $A\colon V\to V^\star$ defines $C:=\norm{A^{-1}}{V^\star\to V}$ and $c:=\min\{c_A,c_A/\norm{A}{V\to V^\star}^2\}$. Using energy norms on $V$ and $V^\star$ with respect to the operator $A$ gives $c=C=1$.
\end{lemma}

\subsection{Radau IIA timestepping schemes}
\label{sec:modCN}
Let $\TT$ denote a time-mesh of the form
\begin{align*}
\TT = \set{T_i=[t_{i-1},t_{i}]}{i=1,\ldots,\#\TT,\,t_0=0<t_1<\ldots<t_{\#\TT}=\tend}.
\end{align*}
We may also index the time steps $t_i$ with the elements $T\in\TT$, i.e., $T=[t_T,t_{T+1}]$. 
Given a finite dimensional vector space $W$, we define the polynomial  spaces
\begin{align}\label{eq:defS2}
\begin{split}
\PP^k(T,W)&:=\set{v\in L^2(T;W)}{v\text{ is a degree $k$ polynomial}},\\
\PP^k(\TT;W)&:=\set{v\in L^2(0,\tend;W)}{v|_T \in \PP^k(T,W)\text{ for all }T\in\TT},\\
\SS^k(\TT;W)&:=\set{v\in H^1(0,\tend;W)}{v|_T \in \PP^k(T,W)\text{ for all }T\in\TT}\quad\text{and}\\
\XX_\TT&:=\XX_\TT(V):=\SS^k(\TT;V).
\end{split}
\end{align}
Define $\Pi_\TT^k\colon L^2(0,\tend;V^\star)\to \PP^k(\TT;V^\star)$ as the $L^2$-orthogonal projection onto piecewise polynomials of order $k$ in time and note $\Pi_T^k(f):=\Pi_\TT^k(f)|_T$ is well-defined for $f\in L^2(T;V^\star)$ and all $T\in\TT$ .
For $k\geq 2$, we propose a time stepping scheme for the heat equation by searching for an approximation $u_\TT\in\XX_\TT$ that satisfies $u_\TT(0)=u_0\in V$ and
\begin{subequations}\label{eq:cnscheme}
\begin{align}\label{eq:cnscheme1}
\int_0^\tend \dual{\partial_t u_\TT}{v}+\dual{Au_\TT}{v}\,ds &= \int_0^\tend \dual{f}{v}\,ds \quad\text{for all }v\in \PP^{k-2}(\TT;V),\\
\dual{\partial_t u_\TT(t_{T+1})}{v}+\dual{Au_\TT(t_{T+1})}{v}&= \dual{\Pi_\TT^{k} f(t_{T+1})}{v} \quad\text{for all }T\in\TT\text{ and }v\in V.\notag
\end{align}
Note that for $k=2$, the first equation is reminiscent of the Crank-Nicolson method, while the second equation comes from the implicit Euler scheme.  We can replace $\Pi_\TT^kf$ by $f$ in the first condition since we test with $\chi_T\in \PP^{k-2}(\TT)$.
It turns out that this scheme is equivalent to the $k$-stage  Radau IIA Runge-Kutta scheme if one replaces $f$ by $\Pi^k_\TT f$ as already pointed out in \cite{andreev_rk}. Radau IIA methods are collocation schemes that collocate at the $k$ right Radau quadrature points. We denote them by $\tau_1^{T,k},\ldots, \tau_k^{T,k} \in T$ in the following. Note that $\tau_k^{T,k}=t_{T+1}$. Thus for $k=2$, an equivalent form of the method reads
\begin{align*}
\begin{split}
    u_\TT(t_{T+1}) &= u_\TT(t_T) + |T|\big(\frac34 k_1 + \frac14 k_2\big),\quad\\
    \text{with $k_1$, $k_2$ solving}\quad 
    &\begin{cases}
    k_1 = (\Pi_T^2f)(t_T+\frac13|T|)-Au_\TT(t_T)-|T| A\big(\frac5{12}k_1-\frac1{12}k_2\big),\\
    k_2 =(\Pi_T^2f)(t_{T+1})-Au_\TT(t_T)-|T| A\big(\frac34 k_1{+}\frac14 k_2\big),
    \end{cases}
    \end{split}
\end{align*}
 as the right Radau quadrature points are $t_T+1/3|T|$ and $t_{T+1}$. 
\subsection{Equivalent Petrov-Galerkin formulation}\label{sec:PG}
For the theoretical considerations below, we want to rewrite~\eqref{eq:cnscheme1} as a Galerkin method. Note that practical computations can be performed with the formulation~\eqref{eq:cnscheme1} above in the classical time stepping sense.

To do this, given $T\in\TT$, define 
\begin{align*}
\TT_T\coloneqq\begin{cases}\{[t_T,t_T+|T|/3],[t_T+|T|/3,t_{T}+2|T|/3],[t_{T}+2|T|/3,t_{T+1}]&\text{for }k=2,\\
\{[t_T,t_T+|T|/2],[t_T+|T|/2,t_{T+1}]\} &\text{for }k>2.
\end{cases}
\end{align*}
as the uniform partition of $T$ into two or three elements and define $\Psi_{T}\in \PP^{k-2}(\TT_T)$ as the representer of the point evaluation functional at $t_{T+1}$, i.e., a function that satisfies $\int_T \Psi_T v\,ds = v(t_{T+1})$ for all $v\in \PP^k(T)$. 
Such a representer exists in all cases due to the following lemma.
\begin{lemma}
For $k\geq 2$, there exists a function $\Psi_T\in \PP^{k-2}(\TT_T)$ with $\int_T \Psi_T v\,ds = v(t_{T+1})$ for all $v\in \PP^k(T)$ and  $\norm{\Psi_T}{L^2(T)}=C_{\rm point} |T|^{-1/2}$ for a constant $C_{\rm point}>0$ that depends only on $k$.
\end{lemma}
\begin{proof}
We consider the case $T=[0,1].$
Due to the finite dimensionality of the problem, it suffices to show that for all $v\in \PP^k(T)\setminus\{0\}$, there exists $\Psi\in \PP^{k-2}(\TT_T)$ with $\int_T \Psi v\,ds\neq 0$:
    We write $v(t) = v_0\prod_{i=1}^n (t-z_i)$ as a product of the real zeros $z_i\in (0,1)$ and the remainder $v_0$. Note that $v_0$ is still a real polynomial in general, but is non-zero in $(0,1)$.
    Clearly, $n\leq k$ and since $\#\TT_T=3$ for $k=2$ and $\#\TT_T=2$ for $k>2$, there exists $T'\in\TT_T$ with $\#\set{j=1,\ldots,n}{z_j\in T'\setminus \{t_{T'},t_{T'+1}\}}\leq k-2$. Define $\Psi\in\PP^{k-2}(\TT_T)$ as $\Psi(t):=\prod_{z_i\in T'\setminus \{t_{T'},t_{T'+1}\}} (t-z_i)$ for $t\in T'$ and $\Psi(t)=0$ else. This satisfies 
    \begin{align*}
    (v\Psi)|_{T'} = v_0\prod_{z_i\notin T'\setminus \{t_{T'},t_{T'+1}\}}(t-z_i)\prod_{z_i\in T'\setminus \{t_{T'},t_{T'+1}\}} (t-z_i)^2.
    \end{align*}
    Note that $v_0$ does not change sign in $T'$. Since this necessarily implies that $v\Psi$ can't change sign in $T'$, we have $\int_T \Psi v\,ds\neq 0$ since $v\Psi$ is not identically zero.
    
    This implies the existence of at least one function $\Psi_{[0,1]}$ with the desired properties and $C_{\rm point}:=\norm{\Psi_{[0,1]}}{L^2([0,1])}$. Scaling shows $\Psi_T(t) = |T|^{-1}\Psi_{[0,1]}((t-t_T)/|T|)$ for all $t\in T$ and hence concludes the proof.
\end{proof}
E.g., for $k=2$, we may compute $\Psi_{[0,1]}$ as
\begin{align*}
\Psi_{[0,1]}(x):=\begin{cases}
    1 & 0\leq x<1/3,\\
    -7/2 & 1/3\leq x<2/3,\\
    11/2 & 2/3\leq x\leq 1.
\end{cases}
\end{align*}    
For all $k\in\N$, we define the weighted point evaluation as $\Phi_{T} := |T|\Psi_{T}$ and
observe $\norm{\Phi_T}{L^2(T)}=C_{\rm point}|T|^{1/2}$. 
In the following, given functions $a\colon [0,\tend]\to \R$ and $b\colon \Omega \to \R$, we denote by $ab$ or $ba$ the function $(t,x)\mapsto a(t)b(x)$.
We may thus define the test space
\begin{align*}
\YY_\TT&:=\YY_\TT(V) :=  \YY_\TT(V)^{\rm point} + \YY_\TT(V)^{\rm mean}\\
&:={\rm span}\set{v\Phi_T}{T\in\TT,\,v\in V}+{\rm span}\set{v\chi}{\chi\in \PP^{k-2}(\TT),\,v\in V}
\end{align*}
and recall the ansatz space $\XX_\TT$ from~\eqref{eq:defS2}.
This results in the scheme: Given $f_\TT:=\Pi_\TT^k f$, find $u_\TT\in\XX_\TT$ such that
\begin{align}\label{eq:cnscheme2}
\int_0^\tend \dual{(\partial_t+A)u_\TT}{v}\,ds + \dual{u_\TT(0)}{w} &= \int_0^\tend \dual{f_\TT}{v}\,ds +\dual{u_0}{w},
\end{align}
for all $ (w,v)\in H\times\YY_\TT.$ Note that this is equivalent to~\eqref{eq:cnscheme1}.
Both the ansatz and the test spaces are nested under mesh refinements with trisections ($k=2)$ or bisections ($k>2$).
\begin{lemma}\label{lem:nested}
    Let $\TT'$ denote a refinement of $\TT$, then $\XX_\TT\subseteq \XX_{\TT'}$. If additionally all elements of $\TT_T$ or their successors are in $\TT'$ for all $T\in \TT\setminus\TT'$, there also holds $\YY_{\TT}\subseteq \YY_{\TT'}$.
\end{lemma}
\begin{proof}
    The ansatz spaces $\XX_\TT$ and
    the $\YY_\TT^{\rm mean}$ parts of the test spaces are nested by definition. 
    Let $v\Phi_T\in \YY_\TT^{\rm point}$. If $T\in \TT'$, there holds $v\Phi_T\in \YY_{\TT'}$. If $T\in \TT\setminus\TT'$, we know $\TT_T\subseteq \TT'$.
    Since $v\Phi_T\in v\PP^{k-2}(\TT_T)\subseteq \YY_{\TT'}^{\rm mean}\subseteq \YY_{\TT'}$, we conclude the proof.
\end{proof}

Note that all formulations of~\eqref{eq:cnscheme} are equivalent to the Radau IIA method and hence well-posed.  E.g., for $k=2$, the method  requires only a local solve with the operator
\begin{align*}
    \widetilde A &=\begin{pmatrix}
        I+\frac{5|T|}{12}A & -\frac{|T|}{12}A\\
        \frac{3|T|}{4}A& I+\frac{|T|}{4}A
    \end{pmatrix}: V \times V \to V^\star \times V^\star.
\end{align*}
The fact that $\widetilde A$ is invertible follows from a simple inf-sup argument, using the coercivity of $A$.

There are multiple possibilities in writing Radau IIA schemes as Petrov--Galerkin schemes. We give the following alternative with non-nested test-spaces introduced in~\cite{andreev_rk}. Define $\widetilde\Psi_T \in \PP^k(T)$ as the representer of the point-evaluation at $t_{T+1}$, i.e. the unique function in $\PP^k(T)$ satisfying $\int_T \widetilde\Psi_T v \ ds = v(t_{T+1})$ for all $v \in \PP^k(T)$. Since Radau nodes of order $k+1$ are exact up to polynomials of order $2k$, we see that
$\int_T v(t) \prod_{i=1}^k(t-\tau_i^{T,k+1})\,dt= \omega_{k+1}\prod_{i=1}^k(t_{T+1}-\tau_i^{T,k+1}) v(t_{T+1})$, where $\omega_{k+1}$ denotes the Radau weight corresponding to $\tau_{k+1}^{T,k+1}$. This shows $\widetilde\Psi_T(t) \sim \Pi_{i=1}^k(t-\tau_i^{T,k+1})$. Denote $\widetilde\Phi_T \coloneqq |T|\widetilde\Psi_T$. With
\begin{align*}
    \widetilde\YY_\TT :={\rm span}\set{v\widetilde\Phi_T}{T\in\TT,\,v\in V}+{\rm span}\set{v\chi}{\chi\in \PP^{k-2}(\TT),\,v\in V},
\end{align*}
the method reads: Find $u_\TT \in \XX_\TT$ such that 
\begin{align}\label{eq:cnscheme3}
\int_0^\tend \dual{(\partial_t+A)u_\TT}{v}\,ds + \dual{u_\TT(0)}{w} &= \int_0^\tend \dual{f}{v}\,ds +\dual{u_0}{w},
\end{align}
\end{subequations}
for all $ (w,v)\in H\times\widetilde\YY_\TT$, which is equivalent to~\eqref{eq:cnscheme2}. Note that in this formulation, we may replace $f_\TT$ by $f$, since $\widetilde \YY_\TT$ is in the range of $\Pi^k_\TT$. The equivalence is quantified in the next Lemma.
\begin{lemma}\label{lem:mod_space}
    There is a uniform constant $C>0$ such that for every $\widetilde v \in \widetilde\YY_\TT$ there is a test function $v \in \YY_\TT$ such that 
    \begin{align}\label{eq:same_bilinear}
        \int_0^\tend \dual{(\partial_t+A)u_\TT}{v}\,ds= \int_0^\tend \dual{(\partial_t+A)u_\TT}{\widetilde v}\,ds\quad\text{for all }u_\TT \in \XX_\TT.
    \end{align}    
    and $C^{-1}\norm{\widetilde v}{\YY}\leq \norm{v}{\YY}\leq C \norm{\widetilde v}{\YY}$.
\end{lemma}
\begin{proof}
    One can construct $v$ elementwise for all $T\in\TT$. Let $\widetilde v|_T=v_1 \widetilde \Phi_T+v_2 \chi$ for some $\chi \in \PP^{k-2}(T)$. With $v|_T:=v_1 \Phi_T+v_2 \chi$, we immediately show~\eqref{eq:same_bilinear}. Furthermore, scaling and norm equivalence on finite dimensional spaces show
    \begin{align*}
        \norm{v}{L^2(T,V)}&\simeq \norm{v_1}{V}\norm{\Phi_T}{L^2(T)}+\norm{v_2}{V}\norm{\chi}{L^2(T)}\quad\text{and},\\
          \norm{\widetilde v}{L^2(T,V)}&\simeq \norm{v_1}{V}\norm{\widetilde\Phi_T}{L^2(T)}+\norm{v_2}{V}\norm{\chi}{L^2(T)}.
    \end{align*}
    Thus, $\norm{\widetilde v}{L^2(T,V)}\simeq \norm{v}{L^2(T,V)}$ follows immediately from $0<\norm{\Phi_T}{L^2(T)}/ \norm{\widetilde\Phi_T}{L^2(T)}<\infty$, which depends only on $k$.
    This concludes the proof.
\end{proof}
\subsection{Error estimator in time}
As $V$ is assumed to be finite dimensional
we immediately have that $u_0\in V$.
We use a residual-based error estimator defined as
\begin{align*}
    \eta_{\TT}^2&=\sum_{T \in \TT}\eta_{\TT}(T)^2, \\
    \eta_{\TT}(T)^2&=|T|^2 \norm{\partial_t f-\partial_t^2u_{\TT} -\partial_t Au_{\TT}}{L^2(T,V^\star)}^2,
\end{align*}
where $u_{\TT}$ solves~\eqref{eq:cnscheme}. 
The fact that this estimator is an upper bound for the error and, moreover, fits into the framework of adaptive mesh refinement, relies on the inf-sup stability of the time stepping (shown below) and arguments from~\cite{infsup}. We note that the data approximation $f_\TT$ does not enter the error estimator $\eta_\TT$ and we also do not need an explicit data oscillation term for computation.
An inverse inequality, a Poincar\'e estimate, and Lemma~\ref{lem:ortho} below show that the modified estimator
\begin{align*}
    \widetilde \eta_\TT(T)^2 := \norm{f-\partial_tu_{\TT} -Au_{\TT}}{L^2(T,V^\star)}^2 +|T|^2\norm{\partial_t (1-\Pi_T^k)f}{L^2(T;V^\star)}^2
\end{align*}
satisfies $\eta_\TT(T)\simeq \widetilde \eta_\TT(T)$ for all $T\in\TT$ and universal equivalence constants. Thus, $\widetilde \eta_\TT(T)$ can be used equivalently in Algorithm~\ref{alg:timeadapt} below, while still being optimal in the sense of Theorem~\ref{thm:optimal_timestepping}.

\subsection{The adaptive algorithm}
\label{sec:adaptive_algorithm}
The adaptive time stepping algorithm (Algorithm~\ref{alg:timeadapt}) uses the standard adaptive loop known from stationary mesh refinement and selects elements for refinement by D\"orfler marking. The only caveat is that, for $k=2$, we require that if an element $T$ gets refined, it is split into three equal parts $\TT_T=\{T_1,T_2,T_3\}$. 
Moreover, we require that the mesh elements do not grow too fast in size in forward time direction, i.e., there exists $G_0>1$ with 
\begin{align}\label{eq:graded}
   |T_{i+1}|/|T_i| \leq G_0 \quad\text{for all }T_i, T_{i+1} \in\TT.
\end{align}
Such a restriction is not uncommon even in classical adaptive time stepping literature.
We note that Theorem~\ref{thm:optimal_timestepping} below shows that the adaptive algorithm is optimal even with respect to the unconstrained class of meshes without grading.
The algorithm expects the user to input initial time steps $\TT_0$,
the marking parameter $0<\theta<1$, and the grading parameter
\begin{align*}
    g_0\geq |T_{i+1}^0|/|T_i^0| \quad \text{ for all } T_i^0, T_{i+1}^0 \in \TT_0.
\end{align*}
For the optimality results below, we require $\theta$ to be sufficiently small. We will show in Lemma~\ref{lem:closure} below, that the refinement step in Algorithm~\ref{alg:timeadapt} ensures the grading condition~\eqref{eq:graded} for all $\ell\in\N$ with $G_0=3g_0$ for $k=2$ and $G_0=2g_0$ for $k>2$.
\begin{algorithm}[H]
\caption{Adaptive time stepping}\label{alg:timeadapt}
\begin{algorithmic}[1]
\For{$\ell=0,1,2,\ldots$}
    \State Solve~\eqref{eq:cnscheme} to obtain $u_\ell:=u_{\TT_\ell}$
    \State Compute $\eta_\ell(T):= \eta_{\TT_\ell}(T)$ for all $T\in\TT_\ell$
    \State Find set of minimal cardinality $\MM_{\ell}\subseteq \TT_\ell$ such that
    \begin{align*}
        \sum_{T\in\MM_{\ell}}\eta_\ell(T)^2\geq \theta \eta_\ell^2
    \end{align*}
    \State Set $\widehat \MM_\ell = \MM_{\ell}$
    \Do
    \State Define $\UU:=\set{T_{i+1}\in\TT_\ell\setminus\widehat\MM_\ell}{T_i \in \widehat\MM_{\ell} \text{ and }|T_{i+1}|/|T_i|>g_0}$
    \State Set $\widehat\MM_{\ell}:=\widehat\MM_{\ell}\cup \UU$
    \doWhile{$\UU\neq \emptyset$}
    \State Set $\TT_{\ell+1}:=\bigcup_{T\in\widehat{\MM_\ell}}\TT_T\cup (\TT_\ell\setminus\widehat\MM_\ell)$.
\EndFor
\end{algorithmic}
\end{algorithm}
\section{Optimality of the adaptive time stepping}
In this section, we prove the following main result of this work.
Let $\TT_0$ be some uniform initial mesh, we denote by $\T$ the set of all meshes that can be obtained via recursive bisections (for $k>2$) and trisections (for $k=2$) starting from $\TT_0$ (without any grading). 

\begin{theorem}\label{thm:optimal_timestepping}
    Let the initial mesh $\TT_0$ be uniform. Moreover, let the marking parameter $0<\theta\leq 1$ be sufficiently small. Given $s>0$, define
    \begin{align}\label{eq:approxclass}
        C_{\rm approx}:=\sup_{N\in\N}\inf_{\substack{\TT\in\T\\\#\TT-\#\TT_0\leq N}} \eta_\TT N^s.
    \end{align}
    The adaptive algorithm produces a sequence of meshes $\TT_\ell$ and solutions $u_\ell\in\XX_{\TT_\ell}$ that satisfy
    \begin{align*}
       \eta_\ell\leq C_{\rm opt}C_{\rm approx} (\#\TT_\ell)^{-s}\quad\text{for all }\ell\in\N,
    \end{align*}
    where the constant $C_{\rm opt}$ is independent of $\ell$, $u$, and the dimension of $V$.
\end{theorem}
\begin{remarkN}
    Note that the optimality result of Theorem~\ref{thm:optimal_timestepping} can equivalently be formulated in terms of the \emph{total error}
    \begin{align*}
        \norm{u-u_\ell}{\XX} + {\rm osc}_{\TT_\ell}(f) \simeq \eta_\ell,
    \end{align*}
    where the equivalence and the data oscillation term ${\rm osc}_\TT(f)$ is from Lemma~\ref{lem:reliability} below.
\end{remarkN}

\begin{remarkN}\label{rem:dualnorm}
The natural norm for the parabolic problem is $\norm{\cdot}{L^2(0,\tend;H^1_0(\Omega))}+\norm{\partial_t(\cdot)}{L^2(0,\tend;H^{-1}(\Omega))}$, while our convergence result in Theorem~\ref{thm:optimal_timestepping} uses the $\XX$-norm which replaces $H^{-1}(\Omega)$ with $V^\star$. The critical equivalence constant $\norm{\partial_t u}{L^2(0,\tend;V^\star)}\simeq \norm{\partial_t u}{L^2(0,\tend;H^{-1}(\Omega))}$ is given by the $H^1(\Omega)$-stability constant of the $L^2(\Omega)$-orthogonal projection onto $V$, see~\cite{TV2016}, which may depend on the dimension of $V$ and thus reintroduce a CFL-like condition into our analysis. However, if $V$ is a standard polynomial FEM space based on a uniform mesh, this constant is well-known to be independent of the mesh-size.
    For adaptively refined meshes created by newest-vertex bisection, the recent works~\cite{DieGrading1,DieGrading2} show the stability for a practically relevant range of polynomial degrees independently of the mesh-size. In these cases, all results involving the $\XX$ norm, which depends on the discrete subspace $V$, still hold with $H^{-1}(\Omega)$ instead of $V^\star$ in the norm definition.
\end{remarkN}
\subsection{Inf-sup stability}
The Radau IIA schemes above can be seen as an extension of the standard Crank-Nicolson scheme, which is inf-sup stable only under a CFL condition~\cite{infsup}. The reason for this is that the Crank-Nicolson scheme is not $L$-stable and hence does not \emph{see} $\mathcal{O}(1)$-amplitude oscillations over the whole time interval (it is easy to construct a function in $\SS^1(\TT;V)$ that is $\mathcal{O}(1)$ in the maximum norm but has vanishing integral mean on each element). The Radau IIA schemes on the other hand are $L$-stable. One could argue that the non-symmetric (from the element in time point of view) modification with the point evaluation at $t_{T+1}$ will dampen such oscillations sufficiently fast. This behavior is quantified in the next theorem. While the implicit Euler scheme would have the same property, we could not find a way to realize it as an inf-sup stable Galerkin method with ansatz and test spaces that are nested under mesh refinement.

\begin{theorem}\label{thm:infsup}
Let $V$ be defined as above and let $\TT$ be a moderately graded time-mesh with $g_0$ from~\eqref{eq:graded}. Then, there exists $c_0>0$ such that
\begin{align*}
\inf_{u\in\XX_{\TT}\setminus\{0\}}\sup_{(v,w)\in (\YY_\TT\times H)\setminus\{0\}}&\frac{\int_0^\tend \dual{(\partial_t + A)u}{v}\,ds +\dual{u(0)}{w}}{\norm{u}{\XX}(\norm{v}{\YY}+\norm{w}{H})}\\
&\qquad \qquad + \frac{|T_1|^{\frac12}\norm{u(0)}{V}}{\norm{u}{\XX}} \geq c_0>0.
\end{align*}
The constant $c_0$ is independent of $\TT$  and $T_1$ is the first time-interval in $\TT$. Moreover, the result also holds when replacing $\YY_\TT$ by $\widetilde\YY_\TT$.
\end{theorem}
\begin{proof}
    We first note that the result follows from the following inf-sup stability 
    \begin{align*}
\inf_{u\in\XX_{\TT}\setminus\{0\}}\sup_{(v,w)\in (\widetilde\YY_\TT\times H)\setminus\{0\}}&\frac{\int_0^\tend \dual{(\partial_t + A)u}{v}\,ds +\dual{u(0)}{w}}{\norm{u}{\XX}(\norm{v}{\YY}+\norm{w}{H})}\\
&\qquad\qquad+ \frac{|T_1|^{\frac12}\norm{u(0)}{V}}{\norm{u}{\XX}} \geq \widetilde{c_0}>0.
\end{align*}
with respect to the modified test space $\widetilde\YY_\TT$ due to Lemma~\ref{lem:mod_space}. To show this inf-sup stability, let $u \in  \XX_\TT \setminus \{0\}$. Denote the elementwise Lagrange interpolation operator at the $k$ right Radau quadrature points $\tau_1^T,\ldots,\tau_k^T$ by $I:\PP^{k}(\TT,V)\to \PP^{k-1}(\TT,V)$.
Following the analysis in \cite[Theorem 3.3]{andreev_rk} and \cite[Section 3.2.2]{andreev_rk} one gets
\begin{align*}
    \Bigg(\sup_{(v,w)\in (\widetilde\YY_\TT\times H)\setminus\{0\}}&\frac{\int_0^\tend \dual{(\partial_t + A)u}{v}\,ds +\dual{u(0)}{w}}{\norm{v}{\YY}+\norm{w}{H}}\Bigg)^2 +|T_1|\norm{u(0)}{V}^2 \\
    &\gtrsim \norm{\partial_t u}{L^2(0,\tend;V^\star)}^2+\norm{I u}{L^2(0,\tend,V)}^2 + |T_1|\norm{u(0)}{V}^2.
\end{align*}
Note that we do not require the self-duality constant defined in~\cite[Equation~27]{andreev_rk} as we use the discrete dual norm, see also Remark~\ref{rem:dualnorm}.
Following Section 3.4 in \cite{andreev_rk}, one gets
\begin{align*}
    \norm{u-Iu}{L^2(T_i;V)}^2\leq \frac{2k}{3}\Big(\frac{|T_{i}|}{|T_{i-1}|}\norm{Iu}{L^2(T_{i-1};V)}^2+\norm{Iu}{L^2(T_{i};V)}^2\Big)
\end{align*}
for all $T_i\in\TT$ with $i>1$. For $T_1$ one obtains
\begin{align*}
     \norm{u-Iu}{L^2(T_1;V)}^2\leq |T_1| \norm{(u-Iu)(0)}{V}^2\lesssim |T_1|\norm{u(0)}{V}^2+ \norm{Iu}{L^2(T_1;V)}^2.
\end{align*}
This implies
\begin{align*}
    \norm{u}{L^2(0,\tend;V)}^2 &\lesssim \norm{Iu}{L^2(0,\tend;V)}^2+ \sum_{T_i\in\TT} \norm{u-Iu}{L^2(T_i;V)}^2
    \\
    &\lesssim kG_0
    \norm{Iu}{L^2(0,\tend;V)}^2 + |T_1|\norm{u(0)}{V}^2
\end{align*}
and hence shows
\begin{align*}
    \norm{\partial_t u}{L^2(0,\tend;V^\star)}^2+\norm{I u}{L^2(0,\tend,V)}^2 + |T_1|\norm{u(0)}{V}^2\gtrsim \norm{u}{\XX}^2,
\end{align*}
where the implied constant depends only on the grading $G_0$ and the polynomial degree $k$. This concludes the proof.
\end{proof}
\begin{remarkN}
    The extra initial data term in the inf-sup stability of Theorem~\ref{thm:infsup} cannot be avoided. The reason for this is that an initial condition with large $V$-norm will be present for at least the first time interval as that is the minimum time required to dampen large eigenvalues (as opposed to the continuous case, where the $V$-norm of the initial condition does not appear).
\end{remarkN}

The inf-sup stability implies the C\'ea lemma.
\begin{corollary}\label{cor:cea}
    Under the assumptions of Theorem~\ref{thm:infsup},  there is a unique solution $u_\TT$ of~\eqref{eq:cnscheme} which satisfies 
    \begin{align*}
       \norm{u-u_\TT}{\XX}\leq \sqrt{2}\frac{\norm{A}{V\to V^\star}}{c_0} \inf_{\substack{v_\TT\in \XX_\TT\\v_\TT(0)=u_\TT(0)}}\norm{u-v_\TT}{\XX}.
    \end{align*}
\end{corollary}
\begin{proof}[Proof]
Theorem \ref{thm:infsup} immediately gives a unique solution $\widetilde u_\TT$ of the homogeneous problem  i.e.,~\eqref{eq:cnscheme} with zero initial condition and $\widetilde f:= f-Au_0\in H^1(0,\tend;V^\star)$ instead of $f$.  It follows that $\widetilde u_\TT+u_0$ solves \eqref{eq:cnscheme}. Theorem \ref{thm:infsup} shows further, that this solution must be unique.
To see the C\'ea estimate, we consider the test space $\widetilde \YY_\TT$ used in \eqref{eq:cnscheme3}.
As using $\widetilde\YY_\TT$ allows to drop the $L^2$-projection in the right-hand side,
we now use the nice idea from~\cite[Theorem~2]{cea}:
Theorem~\ref{thm:infsup} (which also applies to $\widetilde\YY_\TT$ as a test space) shows that the Galerkin projection $P_\TT\colon \XX_0:=\set{v\in\XX}{v(0)=0}\to \XX_{\TT,0}:=\set{v\in\XX_\TT}{v(0)=0}$ is bounded in the sense $\norm{P_\TT}{\XX_0\to \XX_0}\leq \sqrt{2}\norm{A}{V\to V^\star}/c_0$. Since $\norm{1-P_\TT}{\XX_0\to \XX_0}=\norm{P_\TT}{\XX_0\to \XX_0}$ for any bounded projection on a Hilbert space, we conclude $\norm{u-u_\TT}{\XX} = \norm{(1-P_\TT)(u-v_\TT)}{\XX} \leq \sqrt{2}\norm{A}{V\to V^\star}/c_0 \norm{u-v_\TT}{\XX}$ for all $v_\TT\in\XX_\TT$ with $v_\TT(0)=u_\TT(0)=u(0)$ and thus the proof.
\end{proof}
\begin{remarkN}
Note that the restriction $v_\TT(0)=u_\TT(0)$ in Corollary~\ref{cor:cea} is necessary, as one can see from the simple case $A=\lambda$, $f=0$, $u_0=1$, and $V=H=\R$ for sufficiently large $\lambda$, where $v_\TT=0$ is a much better approximation to $u(t)=e^{-\lambda t}$ than any function $v_\TT\in\XX_\TT$ with $v_\TT(0)=1$.
\end{remarkN}

\section{Properties of the error-estimator}\label{sec:errest}
We aim to show the sufficient properties from~\cite{carstensen_2014} of the error estimator for the optimality of the adaptive algorithm.
Recall the definitions $\T$ (meshes obtained from $\TT_0$ via recurrent trisection $k=2$ or bisection $k>2$) and $\T^{\rm grad}\subset \T$ (meshes that additionally respect the grading condition~\eqref{eq:graded}).

We first require a small separation property that allows us to avoid an extra data oscillation term.
\begin{redlemma}\label{lem:ortho}
    Given $f\in H^1(0,\tend;V^\star)$ and $w\in \PP^{k-1}(\TT;V^\star)$, there holds
    \begin{align*}
        \norm{\partial_t(1-\Pi_\TT^k)f}{L^2(T;V^\star)}^2 +  \norm{w}{L^2(T;V^\star)}^2 \leq C\norm{\partial_t(1-\Pi_\TT^k)f+w}{L^2(T;V^\star)}^2
    \end{align*}
    for all $T\in\TT$. The constant $C>0$ is universal.
\end{redlemma}
\begin{proof}[Proof]
A scaling argument reduces the situation to the case $T=[0,1]$.
    Assume the statement is wrong. Then, we find sequences $f_n\in  H^1(T;V^\star)$ and $w_n\in \PP^{k-1}(T;V^\star)$ such that 
    $\norm{w_n}{L^2(T,V^\star)}+\norm{f_n}{H^1(T;V^\star)}=1$, $f_n\perp_{L^2(T;V^\star)} \PP^k(T;V^\star)$, and
    \begin{align*}
        \frac{1}{n}\geq \frac{1}{n}\big(\norm{\partial_t f_n}{L^2(T;V^\star)}^2 +  \norm{w_n}{L^2(T;V^\star)}^2\big) >\norm{\partial_tf_n+w_n}{L^2(T;V^\star)}^2.
    \end{align*}
    This implies for subsequences (that we do not distinguish notationally) $f_n\rightharpoonup f \in H^1(T;V^\star)$ with $f\perp_{L^2(T;V^\star)} \PP^k(T;V^\star)$ and $f_n\to f\in L^2(T;V^\star)$.
    Moreover, $\partial_t f_n +w_n\to 0$ in $L^2(T;V^\star)$ and therefore $w_n\rightharpoonup w\in L^2(T;V^\star)$ and $w=-\partial_t f$. As $\PP^{k-1}(T,V^\star)$ is weakly closed it holds that $w \in \PP^{k-1}(T,V^\star)$, and therefore $f = w = 0$.
    Since $\PP^k(T;V^\star)$ is finite dimensional, we even have $\norm{w_n}{L^2(T;V^\star)}\to 0$ and thus $\norm{\partial_t f_n}{L^2(T;V^\star)}\to 0$. This yields the contradiction $1=\norm{w_n}{L^2(T,V^\star)}+\norm{f_n}{H^1(T;V^\star)}\to 0$.
\end{proof}

The following lemma shows that the error estimator is an upper (also in a discrete sense) and lower bound for the error. The upper bounds are an important ingredient for the optimality Algorithm~\ref{alg:timeadapt}, while the lower bound is required to show that the approximation class defined in~\eqref{eq:approxclass} can equivalently be defined in terms of the total error $\norm{u-u_\TT}{\XX}+{\rm osc}_\TT(f)$ instead of $\eta_\TT$.
\begin{lemma}\label{lem:reliability}
The estimator is reliable and efficient in the sense
    \begin{align}\label{eq:rel}
        C_{\rm eff}^{-1}\eta_\TT\leq \norm{u-u_\TT}{\XX}+{\rm osc}_\TT(f) \leq C_{\rm rel} \eta_{\TT},
    \end{align}
    where the data oscillations are defined by
    \begin{align*}
        {\rm osc}_\TT(f)^2:=\sum_{T\in\TT}|T|^2\norm{\partial_t (1-\Pi_\TT^k)f}{L^2(T;V^\star)}^2
    \end{align*}
    for $\TT \in \T$.
     For a refinement $\widehat\TT \in \T^{\rm grad}$ of $\TT \in \T$, there holds discrete reliability,
     \begin{align}\label{eq:discrel}
         \norm{u_{\widehat\TT}-u_\TT}{\XX}^2 \leq C_{\rm drel}\sum_{T \in \TT \setminus \widehat\TT} \eta_{\TT}(T)^2.
     \end{align}
    The constants $C_{\rm eff}$, $C_{\rm rel}$, $C_{\rm drel}>0$ do not depend on $\TT$ or $\widehat\TT$.
\end{lemma}
\begin{proof}
We start with discrete reliability:
  By Theorem \ref{thm:infsup}, we have that
  \begin{align*}
     c_0 \norm{u_{\widehat\TT}-u_\TT}{\XX} \leq \sup_{ \widehat{v} \in \YY_{\widehat\TT}\setminus\{0\}}\frac{\int_0^\tend \dual{(\partial_t + A)(u_{\widehat\TT}-u_\TT)}{\widehat{v}}\,dt}{\norm{\widehat{v}}{\YY}},
  \end{align*}
  as $u_{\widehat\TT}(0)=u_{\TT}(0).$ Since, on non-refined elements $T \in \widehat\TT \cap \TT$, we have $\int_T \dual{(\partial_t + A)(u_{\widehat\TT}-u_\TT)}{\widehat{v}}_H\,dt=0$, we get
  \begin{align*}
      c_0 \norm{u_{\widehat\TT}-u_\TT}{\XX} &\leq \sup_{ \widehat{v} \in \YY_{\widehat\TT}\setminus\{0\}}\frac{\sum_{T \in \TT \setminus \widehat\TT}\int_T \dual{f_{\widehat\TT}-(\partial_t + A)u_\TT}{\widehat{v}}\,dt}{\norm{\widehat{v}}{\YY}} \\
      &\leq    \Big(\sum_{T \in \TT \setminus \widehat\TT}\norm{f_{\widehat\TT}-\partial_t u_{\TT}-A u_{\TT}}{L^2(T,V^\star)}^2\Big)^{1/2}.
  \end{align*}
  From~\eqref{eq:cnscheme}, we see that $f_{\widehat\TT}-\partial_t u_\TT -A u_\TT$ has vanishing integral mean on each $T \in \TT $. Hence, a Poincaré inequality shows
  \begin{align*}
      \norm{f_{\widehat\TT}-\partial_t u_{\TT}-A u_{\TT}}{L^2(T,V^\star)}^2 &\lesssim |T|^2 \norm{\partial_t f_{\widehat\TT}- \partial_t^2u_{\TT}-\partial_t Au_{\TT}}{L^2(T,V^\star)}^2\\
      &\lesssim
      |T|^2 \norm{\partial_t f- \partial_t^2u_{\TT}-\partial_t Au_{\TT}}{L^2(T,V^\star)}^2,
  \end{align*}
  where we use Lemma~\ref{lem:ortho} on each element $T'\in\widehat\TT$ with $T'\subseteq T$ for the last estimate.
  This immediately implies~\eqref{eq:discrel}. Applying the same arguments and using Lemma~\ref{lem:infsupcont} instead of Theorem~\ref{thm:infsup} yields $\norm{u-u_\TT}{\XX}\lesssim \eta_\TT$. 
  Moreover, Lemma~\ref{lem:ortho} with $ w= \partial_t f_\TT - \partial_t^2u_\TT- A\partial_t u_\TT\in \PP^{k-1}(\TT;V^\star)$ implies also $ {\rm osc}_\TT(f)\leq \eta_\TT$.
  To see the converse estimate, we use a standard inverse inequality to estimate
  \begin{align*}
      \eta_\TT^2&\lesssim \sum_{T\in\TT}|T|^2\norm{\partial_t (1-\Pi_\TT^k)f}{L^2(T;V^\star)}^2+|T|^2\norm{\partial_t f_\TT-\partial_t^2  u_\TT-A\partial_tu_\TT}{L^2(T;V^\star)}^2\\
      &\lesssim 
      {\rm osc}_\TT(f)^2+\norm{ f_\TT-\partial_t  u_\TT-Au_\TT}{L^2(0,\tend;V^\star)}^2\\
       &\lesssim 
      \sum_{T\in\TT}\norm{(1-\Pi_\TT^k)f}{L^2(T;V^\star)}^2+{\rm osc}_\TT(f)^2+\norm{u-u_\TT}{\XX}^2.
  \end{align*}
  Another Poincaré estimate on each interval shows  $\sum_{T\in\TT}\norm{(1-\Pi_\TT^k)f}{L^2(T;V^\star)}^2\lesssim {\rm osc}_\TT(f)^2$ and concludes the proof.
\end{proof}
The following lemma is the basis for the so-called \emph{estimator reduction}, see~\cite{carstensen_2014} for details.
\begin{lemma}\label{lem:stab_red}
    Let $\widehat\TT \in \T$ be a refinement of $\TT \in \T$. The error estimator satisfies reduction on  refined elements, i.e.
    \begin{align}\label{eq:red}
        \sum_{T \in \widehat\TT \setminus \TT} \eta_{\widehat\TT}(T)^2 \leq q \sum_{T \in \TT \setminus \widehat\TT} \eta_{\TT}(T)^2 + C \norm{u_{\widehat\TT}-u_\TT}{\XX}^2,
    \end{align}
    where $0<q<1$ and $C>0$ do not depend on $\TT$ or $\widehat\TT$ as well as stability on non-refined elements,i.e.
    \begin{align}\label{eq:stab}
        \Big|\Big( \sum_{T \in \widehat\TT \cap \TT} \eta_{\widehat\TT}(T)^2 \Big)^{1/2}- \Big(\sum_{T \in \widehat\TT \cap \TT} \eta_{\TT}(T)^2 \Big)^{1/2} \Big| \leq C \norm{u_{\widehat\TT}-u_\TT}{\XX}.
    \end{align}
\end{lemma}
\begin{proof}
    Let $T \in \TT \setminus \widehat\TT$ and let $T_1,\ldots,T_n \in \widehat\TT$ such that $T=\bigcup_{i=1}^n T_i$. It holds that $|T_i|\leq |T|/2$. We have for all $\delta>0$ that
    \begin{align*}
        \sum_{i=1}^n&\eta_{\widehat\TT}(T_i)^2=|T_1|^2\norm{\partial_t f - \partial_t^2 u_{\widehat\TT}-A\partial_t u_{\widehat\TT}}{L^2(T,V^\star)}^2\\
        &\leq (1+\delta)\frac{|T|^2}{4}\norm{\partial_t f - \partial_t^2 u_{\TT}-A\partial_t u_{\TT}}{L^2(T,V^\star)}^2\\
        & + 2|T_1|^2(1+\delta^{-1}) \Big(\norm{\partial_t^2 u_{\widehat\TT}-\partial_t^2 u_{\TT}}{L^2(T,V^\star)}^2 +\norm{A\partial_t u_{\widehat\TT}-A\partial_t u_{\TT}}{L^2(T,V^\star)}^2 \Big).
    \end{align*}
    Standard inverse estimates for polynomials yield
    \begin{align*}
        |T_1|^2\norm{\partial_t^2 &u_{\widehat\TT}-\partial_t^2 u_{\TT}}{L^2(T,V^\star)}^2 +|T_1|^2\norm{A\partial_t u_{\widehat\TT}-A\partial_t u_{\TT}}{L^2(T,V^\star)}^2\\
        &\lesssim \norm{\partial_t u_{\widehat\TT}-\partial_t u_{\TT}}{L^2(T,V^\star)}^2 +\norm{A u_{\widehat\TT}-Au_{\TT}}{L^2(T,V^\star)}^2 
        = \norm{u_{\widehat\TT}-u_{\TT}}{\XX(T)}^2,
    \end{align*}
    which yields the result on $T$. Summing up over all $T \in \TT \setminus \widehat\TT$ gives \eqref{eq:red}. The second statement follows analogously.
\end{proof}
We restate the quasi-monotonicity of the error estimator for convenience.
\begin{redlemma}[{\cite[Lemma~3.5]{carstensen_2014} }]\label{lem:qmon}
    Given $\TT \in \T$ and $\TT'\in\T^{\rm grad}$ such that $\TT'$ is a refinement of $\TT$. Then, there holds
    \begin{align*}
        \eta_{\TT'}\leq C\eta_\TT
    \end{align*}
    for a constant $C>0$ that depends only on the constants from Lemmas~\ref{lem:reliability}--\ref{lem:stab_red}.
\end{redlemma}

\subsection{Proof of the main result}

The idea of the following optimality proof is to use the work~\cite{infsup} together with Lemma~\ref{lem:nested} and Theorem~\ref{thm:infsup} to show \emph{general quasi-orthogonality} of the adaptive algorithm.
Lemma~\ref{lem:closure} below shows that our mesh refinement algorithm satisfies the so-called \emph{closure estimate}. 
 Together with other properties of the error estimator shown in Section~\ref{sec:errest},~\cite{infsup} and the arguments in~\cite{carstensen_2014} then show the optimality result in Theorem~\ref{thm:optimal_timestepping}. We start with \emph{general quasi-orthogonality} in the following lemma.

\begin{redlemma}\label{lem:qo}
The discrete solutions of~\eqref{eq:cnscheme} satisfy general quasi-orthogonality 
\begin{align*}
         \sum_{m=\ell}^{\ell+N} \norm{u_{m+1}-u_m}{\XX}^2\leq C(N)\eta_\ell^2\quad\text{for all }\ell,N\in\N
    \end{align*}
    for some function $C(N)=o(N)$ as $N\to \infty$ that only depends on the inf-sup stability constant $c_0$ from Theorem~\ref{thm:infsup} and the operator $A$.
\end{redlemma}
\begin{proof}[Proof]
    We consider the auxiliary bilinear form
    \begin{align*}
        b((u,g),(v,w)):= \int_0^\tend \dual{(\partial_t +A)u}{v} + \dual{g}{w-v}\, ds
    \end{align*}
    for functions $u\in \XX_\TT$ with $u(0)=0$, $g\in \PP^k(\TT;V^\star)\subset L^2(0,\tend;V^\star)$, $w\in \PP^k(\TT;V)\subset \YY$, and $v\in \YY_\TT$.
    We first show that $b(\cdot,\cdot)$ is inf-sup stable.
    Given $(u,g)\in\XX_\TT\times \PP^k(\TT;V^\star)$ with $u(0)=0$, Theorem~\ref{thm:infsup} provides a test function $v\in\YY_\TT$, which we scale to satisfy $\norm{v}{\YY}=\norm{u}{\XX}$.
    We define $w_v\in \PP^k(\TT;V)$ as the unique solution of
    \begin{align*}
        \int_0^\tend \dual{w'}{w_v}\,ds = \int_0^\tend\dual{w'}{v}\,ds \quad\text{for all }w'\in \PP^k(\TT;V^\star).
    \end{align*}
    Testing with $w':= Aw_v$ shows immediately $\norm{w_v}{\YY}\lesssim \norm{v}{\YY}=\norm{u}{\XX}$.
    Thus, $w:= w_v+A^{-1} g\in \PP^k(\TT;V)$ satisfies
    \begin{align*}
        b((u,g),(v,w))&= \int_0^\tend \dual{(\partial_t +A)u}{v} + \dual{g}{w_v-v + A^{-1}g}\, ds\\
        &\gtrsim \norm{u}{\XX}^2 + \norm{g}{L^2(0,\tend;V^\star)}^2\\
        &\gtrsim (\norm{u}{\XX} + \norm{g}{L^2(0,\tend;V^\star)})(\norm{v}{\YY} + \norm{w}{\YY}),
    \end{align*}
    where we used $\norm{w}{\YY}\lesssim  \norm{w_v}{\YY} + \norm{g}{L^2(0,\tend;V^\star)}\lesssim \norm{u}{\XX} + \norm{g}{L^2(0,\tend;V^\star)}$.
    This shows that $b(\cdot,\cdot)$ is inf-sup stable. Thus, we can write the discrete method~\eqref{eq:cnscheme} equivalently as: Find $(\widetilde u_\TT,\widetilde f_\TT)\in \XX_\TT\times \PP^k(\TT;V^\star)$ such that $\widetilde u_\TT(0)=0$ and
    \begin{align*}
        b((\widetilde u_\TT,\widetilde f_\TT),(v,w))=\int_0^\tend\dual{f-Au_0}{w}\,ds\quad\text{for all } (v,w)\in \YY_\TT\times \PP^k(\TT;V),
    \end{align*}
    where $u_\TT = \widetilde u_\TT + u_0$ solves~\eqref{eq:cnscheme} and $\widetilde f_\TT = f_\TT - Au_0$. The result~\cite[Lemma~9]{infsup} applies to this setting and shows (with $\widetilde f_m:=\widetilde f_{\TT_m}$)
    \begin{align*}
        \sum_{m=\ell}^{\ell+N} \norm{\widetilde u_{m+1}-\widetilde u_m}{\XX}^2 &+ \norm{\widetilde f_{m+1}-\widetilde f_m}{L^2(0,\tend;V^\star)}^2\\
        &\leq C(N)\big(\norm{\widetilde u-\widetilde u_\ell}{\XX}^2 + \norm{\widetilde f-\widetilde f_\ell}{L^2(0,\tend;V^\star)}^2\big),
    \end{align*}
    where $\widetilde u =u-u_0\in\XX$, $\widetilde f=f-Au_0$, and $C(N) = o(N)$ as $N\to \infty$. This and Lemma~\ref{lem:reliability} imply immediately
    \begin{align*}
         \sum_{m=\ell}^{\ell+N} \norm{u_{m+1}-u_m}{\XX}^2\lesssim C(N)\big(\norm{ u-u_\ell}{\XX}^2 + {\rm osc}_\ell(f)^2\big)\lesssim C(N)\eta_\ell^2.
    \end{align*}
\end{proof}

We are finally ready to prove the main result.
\begin{proof}[Proof of Theorem~\ref{thm:optimal_timestepping}]
    Lemma~\ref{lem:graded_equiv} shows that we can replace $\T$ by the class of graded meshes $\T^{\rm grad}$ in~\eqref{eq:approxclass} without changing the optimal rate $s>0$.
    The result~\cite[Lemma~4.7]{carstensen_2014} shows that Lemma~\ref{lem:stab_red} implies \emph{estimator reduction}
    \begin{align*}
        \eta_{\ell+1}^2\leq q\eta_\ell^2 + C\norm{u_{\ell+1}-u_\ell}{\XX}^2
    \end{align*}
    for all $\ell\in\N$ and some constants $0<q<1$ and $C>0$.
    Together with \emph{reliability} from Lemma~\ref{lem:reliability}, \emph{quasi-monotonicity} from Lemma~\ref{lem:qmon}, and \emph{general quasi-orthogonality} from Lemma~\ref{lem:qo},
    the result~\cite[Lemma~6]{infsup} shows \emph{linear convergence}
    \begin{align*}
        \eta_{\ell+m}\leq Cq^m\eta_\ell\quad\text{for all }\ell,m\in\N
    \end{align*}
    for some (different) constants $0<q<1$ and $C>0$.
    The remaining assumptions on the mesh-refinement are shown in Appendix~\ref{sec:closure}. Finally,~\cite[Theorem~4.1]{carstensen_2014} concludes the proof.
\end{proof}

\section{Variable space discretizations}
In this section we want to generalize our timestepping scheme to variable space discretizations and provide an inf-sup stability result. The idea was already hinted at in~\cite{andreev_rk}, but it deserves more rigor: We recall the Gelfand triplet $V \subset H \subset V^\star$ with the difference, that we allow possibly infinite dimensional spaces. We recall the time discretization
$\TT = \set{T_i=[t_{i-1},t_{i}]}{i=1,\ldots,\#\TT,\,t_0=0<t_1<\ldots<t_{\#\TT}=\tend}$,
and associate to each node $t_i$ a finite-dimensional subspace $V^h_i \subset V$. Recall the right Radau quadrature points $\tau^{T,k}_1,\ldots,\tau^{T,k}_k$ associated to the interval $T$. We define the spaces
\begin{align}
    \PP^{k,{\rm var}}(T;V^h_{T},&V^h_{T+1}):=\{v \in \PP^k(T,V^h_{T}+V^h_{T+1}): v(\tau^{T,k}_j) \in V^h_{T+1},\, j =1,\ldots,k\}\\
    \XX_\TT^{\rm var}&:=\{v \in H^1(0,\tend;V): v_{| T} \in  \PP^{k,{\rm var}}(T;V^h_{T}, V^h_{T+1}), \text{ for all } T \in \TT \}.
\end{align}
We equip $ \XX_\TT^{\rm var}$ with the norm $\norm{u}{ \XX_\TT^{\rm var}}^2:= \sum_{T \in \TT} \norm{\partial_t u}{L^2(T;(V_{T+1}^h)^\star)}^2 + \norm{u}{L^2(0,\tend;V)}^2.$ 
To define the right test spaces, we reintroduce the elementwise Lagrange interpolation operator at the $k$ right Radau quadrature points $\tau_1^T,\ldots,\tau_k^T$ by $I:\PP^{k}(\TT,V)\to \PP^{k-1}(\TT,V)$. Furthermore define its elementwise $L^2(T)$ adjoint $I^\star\colon \PP^{k-1}(\TT,V) \to \PP^k(\TT,V)$ via
\begin{align*}
     \int_0^\tend (I^\star v) p \ dt = \int_0^\tend v (I p) \ dt\quad \text{ for all } (v,p) \in \PP^{k-1}(\TT,V) \times \PP^k(\TT,\R).
\end{align*}
The testspace is defined as
\begin{align}\label{eq:defYYvar}
\begin{split}
    \YY^{\rm var}_\TT &:= I^\star I \XX^{\rm var}_\TT\\
    &=\{v \in L^2(0,\tend;V): v_{|T} \in {\rm span}\set{v\widetilde\Phi_T}{v\in V_{T+1}^h}+\PP^{k-2}(T,V_{T+1}^h) \}.
    \end{split}
\end{align}
Note that the characterization of $\YY^{\rm var}_\TT$ is shown for fixed spaces in~\cite[Proposition~2.3]{andreev_rk} but does not directly generalize to variable spaces as is done in the following lemma.
\begin{lemma}
    The second equality in~\eqref{eq:defYYvar} is true.
\end{lemma}
\begin{proof}
We note that $I$ and $I^\star$ act locally on each $T\in\TT$.
We define $\pi_{T,k}:=\prod_{i=1}^{k-1}(t-\tau_i^{T,k})$ as the polynomials that vanish on all but the rightmost Radau points. It has been noted in Section~\ref{sec:PG} above that $\pi_{T,k+1} \sim \widetilde\Phi_T$. Moreover, there holds for $v\in \PP^k(T)$ 
\begin{align*}
    \int_T I^\star \pi_{T,k} v\,dt = \int_{T}\pi_{T,k} Iv\,dt \sim (Iv)(t_{T+1}) = v(t_{T+1}) \sim \int_T \widetilde\Phi_T v\,dt.
\end{align*}
This shows $I^\star \pi_{T,k}\sim \widetilde \Phi_T$.
The same duality argument shows for $w\in \PP^{k-2}(T)$ and $v\in\PP^k(T)$ that $ \int_T I^\star w v\,dt =  \int_T w Iv\,dt = \int_T wv\,dt  $, where the last equality follows from the fact that order $k$ Radau nodes are exact on $\PP^{2k-2}(T)$. This shows $I^\star \PP^{k-2}(T) = \PP^{k-2}(T)$.
Finally, we have the decomposition 
 $\PP^{k-1}(T;V_{T+1}^h) = \PP^{k-2}(T;V_{T+1}^h) + \pi_{T,k}V_{T+1}^h$. Since $I\XX_\TT^{\rm var}$ restricted to one interval $T$ is equal to $\PP^{k-1}(T;V_{T+1}^h)$, the arguments above conclude the proof.
\end{proof}

There holds the following inf-sup stability.
\begin{theorem}
    Let $ \XX_\TT^{\rm var}$ and $\YY_\TT^{\rm var}$ be defined as above and let $\TT$ be a moderately graded time-mesh with $g_0$ from~\eqref{eq:graded}. Then, there exists $\widetilde{c_0}>0$ such that
    \begin{align*}
        \inf_{u \in \XX^{\rm var}_\TT \setminus \{0\}} \sup_{(v,w)\in \YY^{\rm var}_\TT \times V^h_0 \setminus \{0\}} &\frac{\int_0^\tend \dual{(\partial_t + A)u}{v}\,ds +\dual{u(0)}{w}}{\norm{u}{ \XX^{\rm var}_\TT}(\norm{v}{\YY}+\norm{w}{H})}\\
&\qquad\qquad+ \frac{|T_1|^{\frac12}\norm{u(0)}{V}}{\norm{u}{ \XX^{\rm var}_\TT}} \geq \widetilde{c_0}>0.
    \end{align*}
\end{theorem}
\begin{proof}
    We rewrite the proof of~\cite[Theorem 3.3]{andreev_rk} in a more direct way.  Given $u\in \XX_\TT^{\rm var}$, define the test function $v_u \in \YY_\TT^{\rm var}$ as
    \begin{align*}
        v_u|_{T} := I^\star(A|_{V_{T+1}^h})^{-1}\Pi_{V_{T+1}^h}\partial_t u + I^\star I u,
    \end{align*}
    where $\Pi_{V_{T+1}^h}$ denotes the $H$-orthogonal projection onto $V_{T+1}^h$ and $A|_{V_{T+1}^h}\colon V_{T+1}^h\to (V_{T+1}^h)^\star$ is the restriction of $A$ to the discrete space.
    It holds that
    \begin{align*}
          &\int_0^\tend \dual{\partial_t u+ Au}{v_u} \ dt= \sum_{T \in \TT} \int_T \dual{\partial_t u+ Au}{I^\star(A|_{V_{T+1}})^{-1}\Pi_{V_{T+1}^h}\partial_t u + I^\star I u} \ dt\\
          &= \sum_{T \in \TT} \int_T \dual{\partial_t u}{(A|_{V_{T+1}})^{-1}\Pi_{V_{T+1}^h}\partial_t u} + \dual{A Iu}{Iu} + 2\dual{\partial_t u}{Iu} \ dt\\
          &\gtrsim \sum_{T \in \TT}\Big(\norm{\partial_t u}{L^2(T;(V_{T+1}^h)^\star)}^2+ \norm{Iu}{L^2(T;V)}^2 + \int_T 2\dual{\partial_t u}{Iu} \ dt\Big).
    \end{align*}
We consider the term $2\int_T \dual{\partial_t u}{Iu}\,dt$. Let $T=[t_{i-1},t_i] \in \TT$, and let $\phi_1^T,\ldots,\phi^T_{m}$ be a $H$-orthonormal basis of $V_{i-1}^h+V_i^h$ with $m={\rm dim}(V_{i-1}^h+V_i^h)$. This gives with the decomposition $u_{|T}(t)=\sum_{j=1}^{m}c_j(t)\phi_j^T$ and hence
\begin{align*}
    \int_T \dual{\partial_t u}{Iu} \ dt = \sum_{j=1}^{m} \int_T \partial_t c_j I c_j \ dt.
\end{align*}
Applying \cite[Lemma 3.1]{andreev_rk} gives
\begin{align*}
     \int_T \partial_t c_j I c_j \ dt \geq  \int_T \partial_t c_j c_j \ dt \quad \text{for all } j =1,\ldots, m.
\end{align*}
Altogether, this shows
\begin{align*}
     \int_0^\tend &\dual{\partial_t u+ Au}{v_u} \ dt+ \norm{u(0)}{H}^2\\
    &\gtrsim
    \sum_{T \in \TT}\Big(\norm{\partial_t u}{L^2(T;(V_{T+1}^h)^\star)}^2+ \norm{Iu}{L^2(T;V)}^2 \ dt\Big) + \norm{u(\tend)}{H}^2.
\end{align*}
Since $\norm{v_u}{\YY}^2 + \norm{u(0)}{H}^2 \lesssim \sum_{T \in \TT}\Big(\norm{\partial_t u}{L^2(T;(V_{T+1}^h)^\star)}^2+ \norm{Iu}{L^2(T;V)}^2 \Big)+\norm{u(\tend)}{H}^2$, we get that
    \begin{align*}
         \sup_{(v,w)\in \YY^{\rm var}_\TT \times V^h_0 \setminus \{0\}} &\frac{\int_0^\tend \dual{(\partial_t + A)u}{v}\,ds +\dual{u(0)}{w}}{(\norm{v}{\YY}+\norm{w}{H})} \\
         &\gtrsim \Big(\sum_{T \in \TT}\norm{\partial_t u}{L^2(T;(V_{T+1}^h)^\star)}^2\Big)^{1/2}
    + \norm{Iu}{L^2(0,\tend;V)} + \norm{u(\tend)}{H}. 
    \end{align*}
We conclude the proof analogously to the proof of Theorem \ref{thm:infsup}.
\end{proof}
\section{Numerical Experiments}
We consider the heat equation, i.e., 
$A = -\Delta$ confined to the physical domain 
$\Omega = (0,1)^2 \subset \mathbb{R}^2$, equipped 
with homogeneous Dirichlet boundary conditions on
$\partial \Omega$ and set $\tend=1$. 
For the space discretization, we consider a conforming,
uniformly shape regular triangulation $\mathcal{E}_h$ of the domain $\Omega$ with mesh-size $h>0$ and consider as well the lowest order Lagrangian finite element space, in the following, denoted by $V:=\mathcal{S}^1_0(\mathcal{E}_h)$, which also satisfies the homogeneous Dirichlet boundary
conditions. 
The computational implementation is conducted
in the {\tt Matlab} library {\tt MooAFEM} \cite{innerberger2023mooafem}. 

\subsection{Fixed Space Mesh \texorpdfstring{$\mathcal{E}_h$}{Eh}}
\label{sec:fixed_space_mesh}
We first study temporal singularities. Two typical cases are relevant here: start-up singularities caused by incompatible initial data, and singularities induced by the right-hand side $f$.
We consider a spatial mesh $\mathcal{E}_h$ with
$h = 7.8 \times 10^{-3}$ and $\text{dim}(\mathcal{S}^1_0(\mathcal{E}_h)) \approx 8 \times 10^{3}$. 
For the initial condition of the semi-discretization,
we consider the $L^2(\Omega)$-orthogonal projection of $u_0$
onto $\mathcal{S}^1_0(\mathcal{E}_h)$. 

Figure~\ref{fig:Test_1_a} portrays the convergence of
both the adaptive and uniform time stepping (orders $k=2,3$),
without any mesh grading procedure, and their corresponding estimator (Figures~\ref{fig:Test_1_a_Error}, \ref{fig:Test_1_a_Error_p3}) together with the timestep size
over the time interval $[0,\tend]$ of the last iteration of the algorithm (Figures~\ref{fig:Test_1_a_Error_Mesh_Size}, \ref{fig:Test_1_a_Error_Mesh_Size_p3}).
In these computations, the initial value and right-hand side are set to $u_0=1$ and $f = 0$ in $\Omega$, respectively.
We observe a strong refinement towards $t=0$ to resolve the start-up singularity. Note that even though we project the non-matching initial condition $u_0=1$ into $\mathcal{S}^1_0(\mathcal{E}_h)$, the large gradient of the result triggers a near-singularity in the solution the becomes even stronger on finer spatial meshes.

Figure~\ref{fig:fsing} shows the same convergence study but for $u_0=0$ and different right-hand sides $f$ that are constant in space but singular in time. Particularly, we consider: $f(t,x)=|t-0.5|^{0.55}$ (top row), $f(t,x)=\max\{t-\pi/5,0\}$ (middle row), and $f(t,x)=\max\{1-\frac{10}{\pi}t,0\}$
Note that all right-hand sides satisfy $t\mapsto f(t,x)\in H^1(0,\tend)$. Differentiation of the equation shows $\partial_t u(0)=f(0)$ and thus $f(0)\notin \mathcal{S}^1_0(\mathcal{E}_h)$ induces a (weaker) start-up singularity which triggers refinement towards $t_0$ in the top and bottom row of Figure~\ref{fig:fsing}.
While the refinement towards the singularity at $t=0.5$ in the top row is very symmetric, the corresponding refinement towards the singularities at $t=\pi/5$ and $t=\pi/10$ show a directional grading that fits very well with~\eqref{eq:graded}, even though we do not enforce it in these experiments: Time steps are allowed to shrink arbitrarily fast but must grow slowly in forward time direction. In the right-hand side plot of Figure~\ref{fig:Test_1}, we plot the step size for $f(t,x)=\max\{t-\pi/5,0\}$ (middle row of Figure~\ref{fig:fsing}) over the distance from the singularity at $t=\pi/5$. We observe that, at least in this experiment, the adaptive algorithm automatically generates meshes that obey~\eqref{eq:graded} for $g_0=1/3$, even though no grading is enforced.

In Figure~\ref{fig:Test_1}, we observe that the algorithm seems to realize a moderate grading automatically, without enforcing it.

\begin{figure}[!ht]
	\centering
	\begin{subfigure}{0.452\linewidth}
		\includegraphics[width=\textwidth]
		{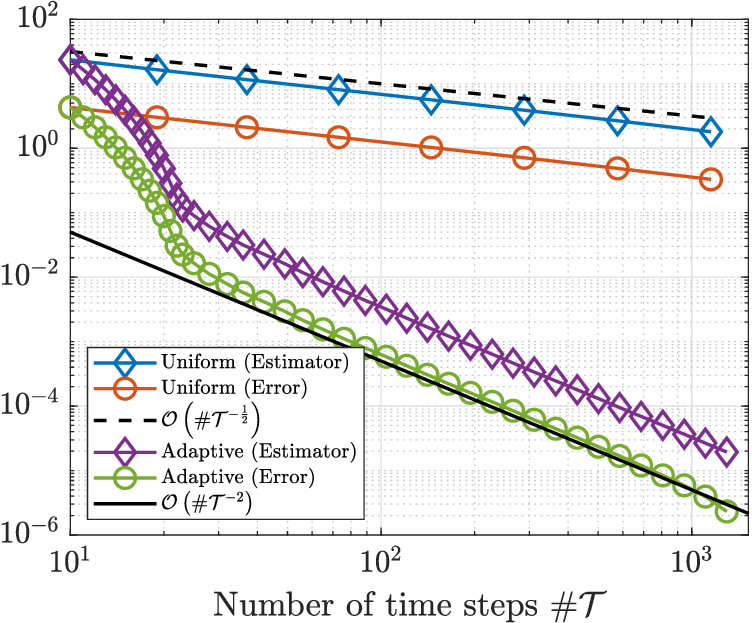}
		\subcaption{
		Errors and estimator (order $k=2$) vs. $\#\TT$}
		\label{fig:Test_1_a_Error}
	\end{subfigure}
	\begin{subfigure}{0.495\linewidth}
		\includegraphics[width=\textwidth]
		{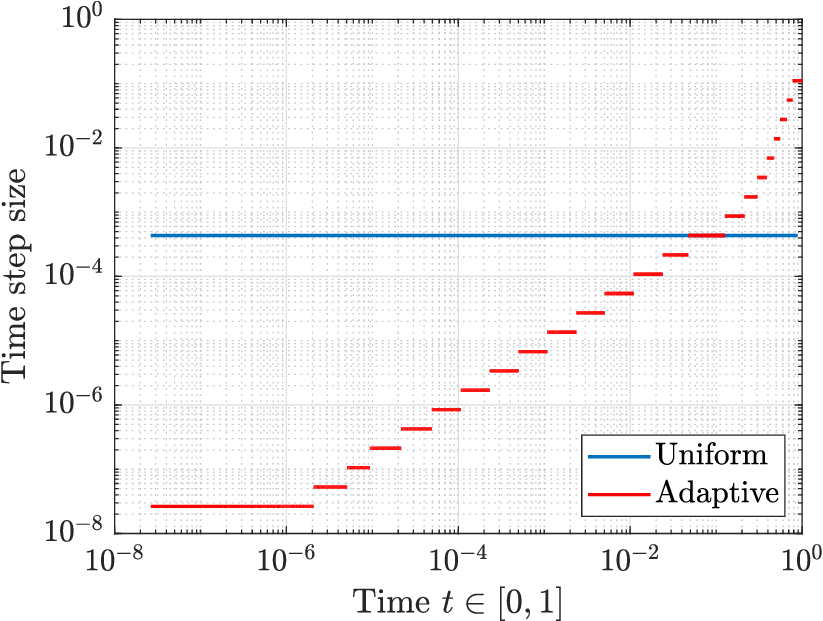}
		\subcaption{Element size (order $k=2$) vs. position in time}
		\label{fig:Test_1_a_Error_Mesh_Size}
	\end{subfigure}
	 \\
	\centering
    \begin{subfigure}{0.452\linewidth}
		\includegraphics[width=\textwidth]
		{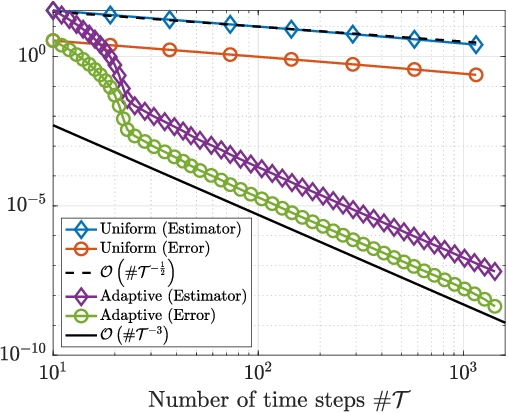}
		\subcaption{
		Errors and estimator (order $k=3$) vs. $\#\TT$}
		\label{fig:Test_1_a_Error_p3}
	\end{subfigure}
	\begin{subfigure}{0.495\linewidth}
		\includegraphics[width=\textwidth]
		{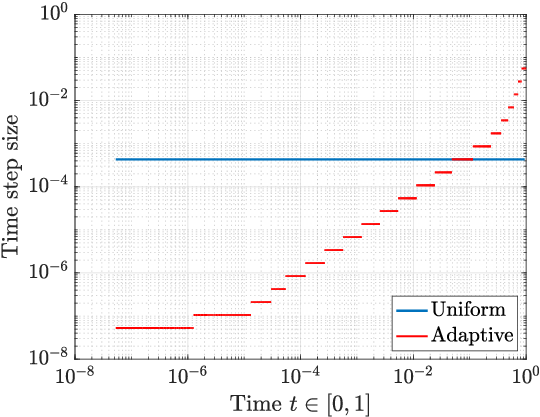}
		\subcaption{Element size (order $k=3$) vs. position in time}
		\label{fig:Test_1_a_Error_Mesh_Size_p3}
	\end{subfigure}
    \\
	\caption{\label{fig:Test_1_a} 
    (a, c) Convergence of the error in the $\mathcal{X}$-norm (comparison with finest approximation) and estimator for adaptive ($\theta = 1/2$) and uniform mesh refinement for orders $k=2,3$.
    (b, d)  Corresponding sizes of local time steps of the last iteration of the adaptive/uniform algorithm plotted
    over their position in the time interval $[0, 1]$. 
    The initial datum and right-hand side are set to $u_0=1$ and $f = 0$ in $\Omega$, respectively.
	}
\end{figure}

\begin{figure}[!ht]
	\centering
		\includegraphics[width=0.4625\textwidth]{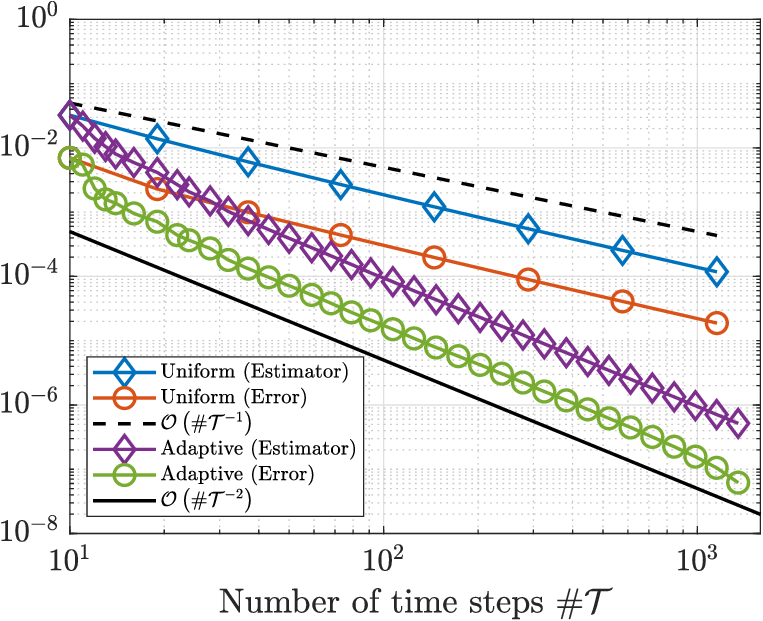}%
		\includegraphics[width=0.49\textwidth]{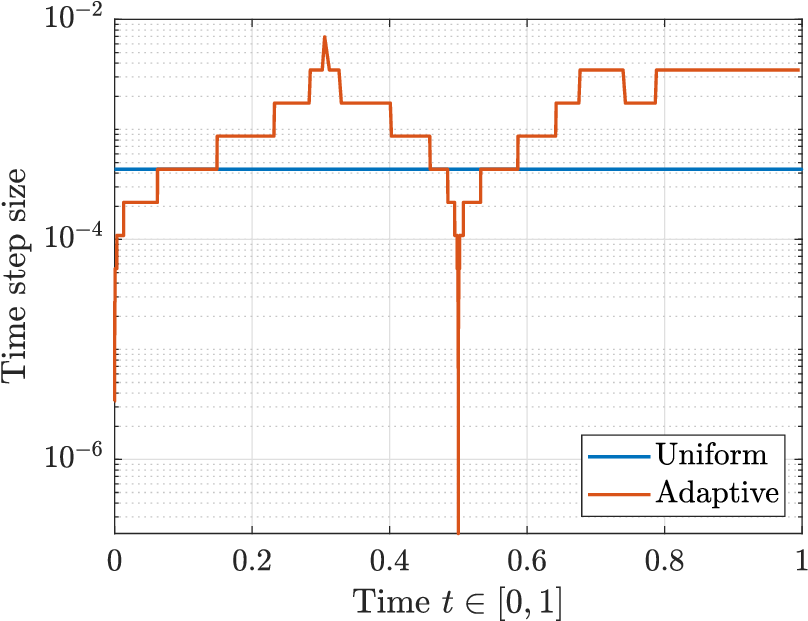}\\%
		\includegraphics[width=0.4625\textwidth]{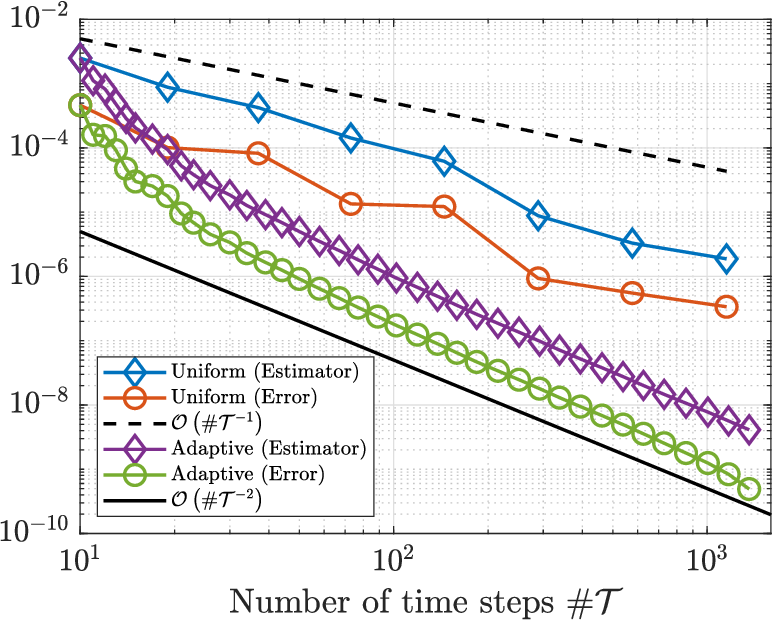}%
		\includegraphics[width=0.49\textwidth]{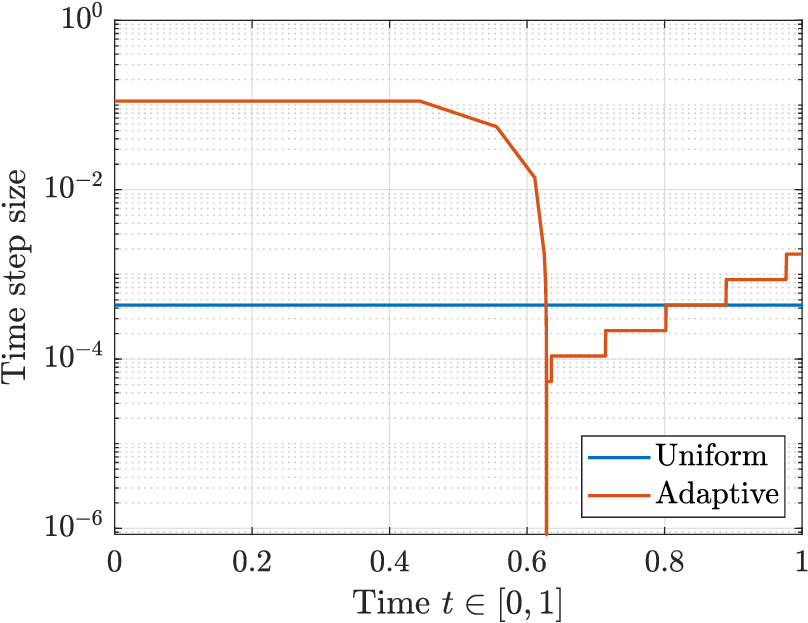}\\%
        \includegraphics[width=0.4625\textwidth]{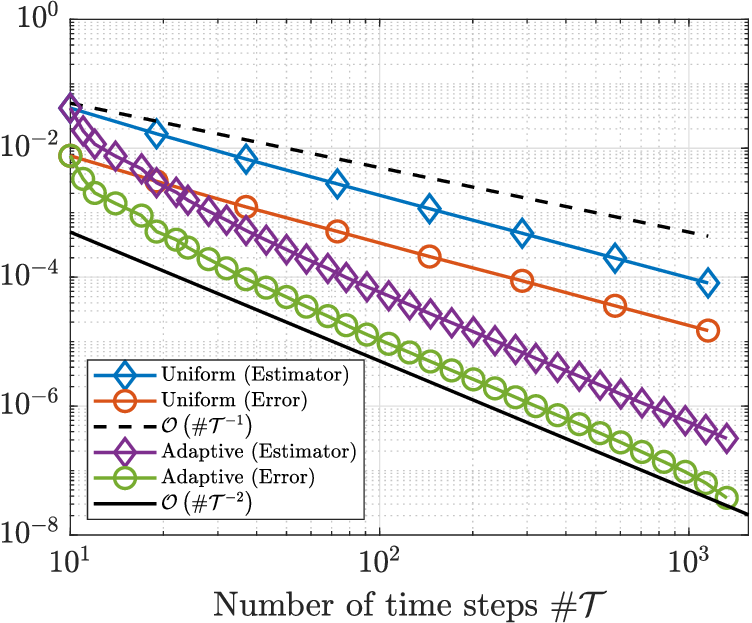}%
		\includegraphics[width=0.49\textwidth]{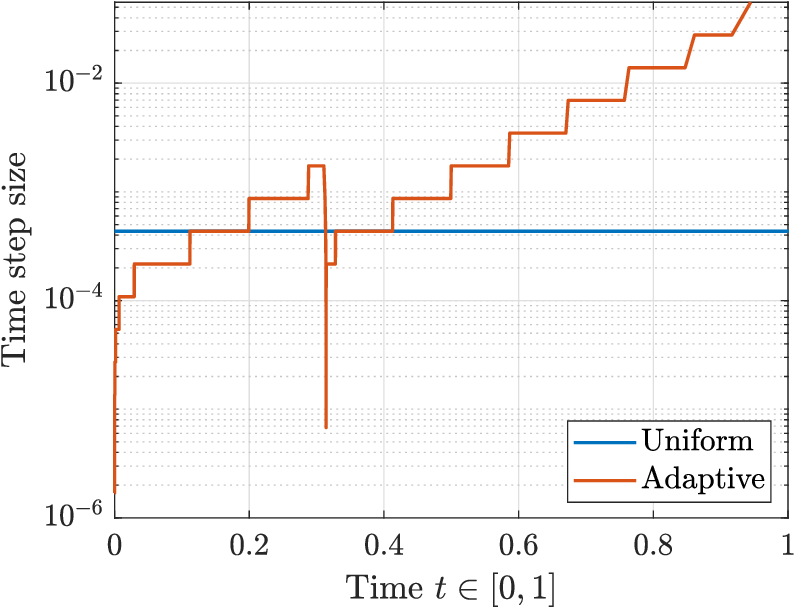}\\%
	\caption{
    Convergence plots for error in the $\mathcal{X}$-norm (w.r.t. reference approximation) and estimator for adaptive ($\theta=1/2$) and uniform mesh-refinement (left column) and local time step sizes of the final adaptive and uniform meshes (right column). The initial condition is $u_0=0$ for all experiments and the right-hand sides are: $f(t,x)=|t-0.5|^{0.55}$ (top row), $f(t,x)=\max\{t-\pi/5,0\}$ (middle row), $f(t,x)=\max\{1-\frac{10}{\pi}t,0\}$ (bottom row).
	}\label{fig:fsing}
\end{figure}

\begin{figure}[!ht]
	\centering
			\includegraphics[width=0.70\textwidth]{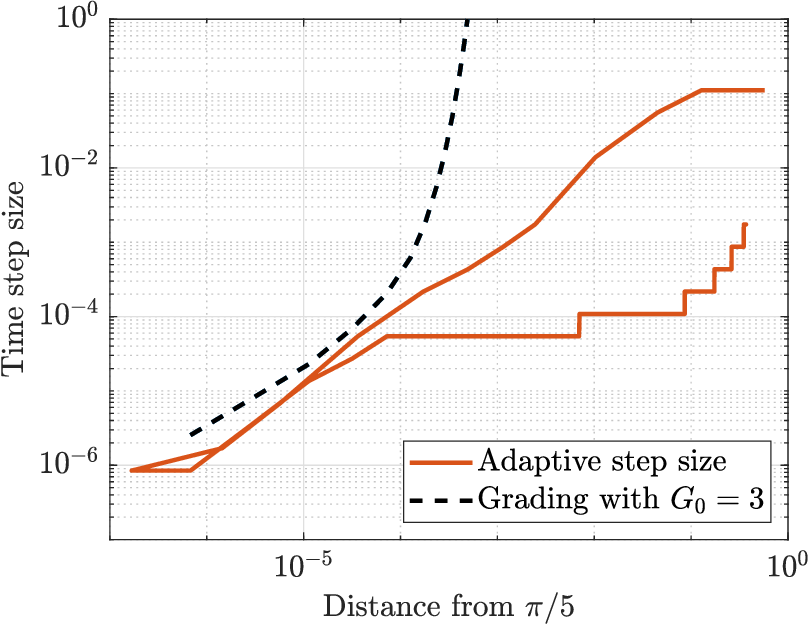}	
	\caption{\label{fig:Test_1} 
   Grading analysis of the experiment in the middle row from Figure~\ref{fig:fsing}: Step size over distance from singularity at $t=\pi/5$ compared with theoretical upper bound for step size from~\eqref{eq:graded} with $G_0=3$.}		
\end{figure}

\subsection{Influence of the Space Discretization}
In Figure~\ref{fig:Test_3_a_Plots}, we compare the
accuracy of the Crank-Nicolson method from~\cite{infsup} and the proposed scheme~\eqref{eq:cnscheme}.
In Figures \ref{fig:Test_3_a_Plots_1}-\ref{fig:Test_3_a_Plots_4}, we compute the errors and estimators for these two methods as we uniformly refine the spatial mesh $\mathcal{E}_h$. More precisely, this results in a sequence of meshes with
$N_h \approx 5 \times 10^2,
2 \times 10^3,8\times 10^3,3.2\times 10^4$, which correspond to mesh sizes $h = 3.12 \times 10^{-2}, 1.56 \times 10^{-2},7.81 \times 10^{-3}, 3.9\times 10^{-3}$.
One can readily observe that the Crank-Nicolson scheme suffers from a pre-asymptotic regime and only 
yields optimal convergence once a CFL condition is fulfilled as predicted by the theory. Furthermore, we observe that the convergence of~\eqref{eq:cnscheme} is oblivious to the
problem's underlying FE discretization.

\begin{figure}[!ht]
	\centering
	\begin{subfigure}{0.465\linewidth}
		\includegraphics[width=\textwidth]
		{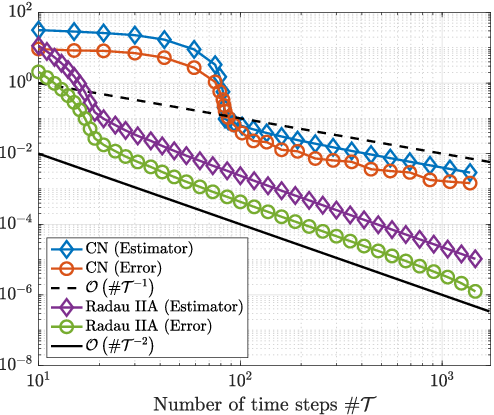}
		\subcaption{
		$N_h \approx 5 \times 10^2$.}
		\label{fig:Test_3_a_Plots_1}
	\end{subfigure}
	\begin{subfigure}{0.465\linewidth}
		\includegraphics[width=\textwidth]
		{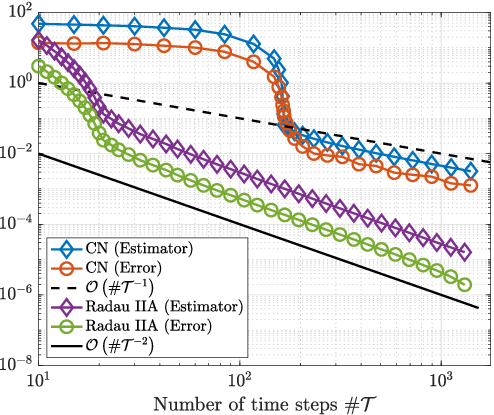}
		\subcaption{
		$N_h \approx 2 \times 10^3$.}
		\label{fig:Test_3_a_Plots_2}
	\end{subfigure}
    \\
	\centering
	\begin{subfigure}{0.465\linewidth}
		\includegraphics[width=\textwidth]
		{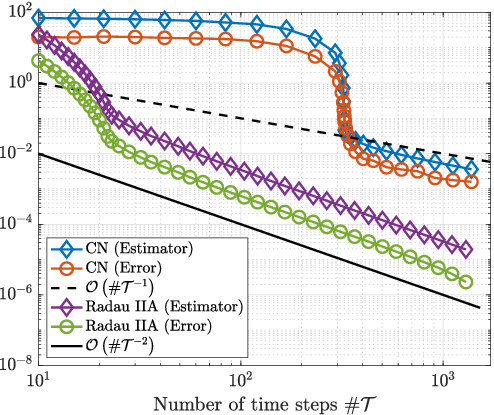}
		\subcaption{
		$N_h \approx 8 \times 10^3$.}
		\label{fig:Test_3_a_Plots_3}
	\end{subfigure}
	\begin{subfigure}{0.465\linewidth}
		\includegraphics[width=\textwidth]
		{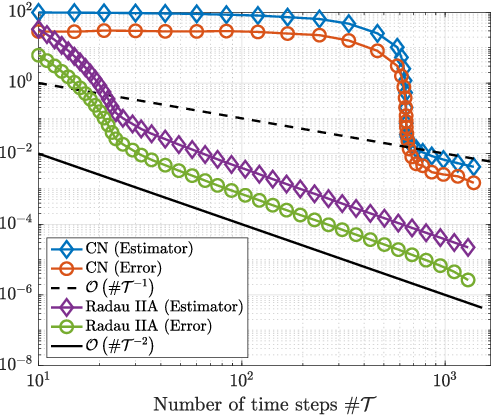}
		\subcaption{
		$N_h \approx 3.2 \times 10^4$.}
		\label{fig:Test_3_a_Plots_4}
	\end{subfigure}
    \\
	\caption{\label{fig:Test_3_a_Plots} 
    Comparison of the Crank-Nicolson (CN) and the proposed scheme~\eqref{eq:cnscheme} (Hybrid) for adaptive ($\theta = 1/2$) mesh refinement and for different values of the total number of degrees of freedom $N_h$.
    The convergence of the error is computed in the $\mathcal{X}$-norm (comparison with the finest approximation). 
    The meshes $\mathcal{E}_h$ are obtained through a successive uniform refinement of a given starting mesh. 
	}
\end{figure}


\appendix

\section{Moderate grading of time meshes}\label{sec:closure}
We first show that Algorithm~\ref{alg:timeadapt} produces meshes that are moderately graded and that the crucial \emph{closure estimate} is satisfied. 
\begin{lemma}\label{lem:closure}
    The refinement procedure from Algorithm~\ref{alg:timeadapt} guarantees moderate gradedness~\eqref{eq:graded} with $G_0=3g_0$ for $k=2$ and $G_0=2g_0$ for $k>2$ and satisfies the closure estimate, i.e., there is a constant $C_{\rm cl}>0$ such that
    \begin{align}\label{eq:closure}
    \#\TT_\ell - \#\TT_0 \leq C_{\rm cl} \sum_{k=0}^{\ell-1} \#\MM_k \quad\text{for all }\ell\in\N.
    \end{align}
\end{lemma}
\begin{proof}
To see~\eqref{eq:graded}, we argue by induction. The initial mesh $\TT_0$ is moderately graded by assumption. Assume that $\TT_\ell$ is moderately graded. The grading~\eqref{eq:graded} can only be violated if an element $T_i$ is refined while its forward neighbor $T_{i+1}$ is not refined. Since the marking procedure ensures that this can only happen if $|T_{i+1}|\leq g_0|T_i|$, the new mesh $\TT_{\ell+1}$ satisfies~\eqref{eq:graded} with $G_0=3g_0$ for $k=2$ and $G_0=2g_0$ for $k>2$.
This concludes the induction.

The closure estimate for the present 1D graded refinement was shown in~\cite{closure1d} with the slight variation that the grading was enforced for all neighbors instead of just the neighbor forward in time. Let $\widetilde \TT_\ell$ denote the sequence of meshes generated by this variation, i.e., with $\widetilde \MM_\ell := \widetilde \TT_\ell\cap \MM_\ell$ in Step~4 of Algorithm~\ref{alg:timeadapt}, and $\widetilde \UU:=\set{T\in\widetilde\TT_\ell\setminus\widehat{\widetilde \MM_\ell}}{\exists \text{neighbor }T'\in \widehat{\widetilde\MM_\ell} \text{ with } |T|/|T'|>g_0}$ instead of $\UU$ in Step~7 of Algorithm~\ref{alg:timeadapt}.

We show by induction that $\widetilde \TT_\ell$ is always a refinement of $\TT_\ell$.
For $\ell=0$, we have $\TT_0=\widetilde \TT_0$ by definition. Assume that $\widetilde \TT_\ell$ is a refinement of $\TT_\ell$.
Let $T\in\TT_\ell\setminus\TT_{\ell+1}$ be an element that is refined in $\TT_\ell$ to generate $\TT_{\ell+1}$.
We consider three cases:

Case~1: Let $T\in\MM_\ell\cap\widetilde \TT_\ell$ in Step~4, which implies $T \in \widetilde \MM_\ell$ and therefore $T$ will be refined to create $\widetilde \TT_{\ell+1}$. 

Case~2: Let $T\in\UU\cap \widetilde \TT_\ell$ in some iteration of Step~7. Since $\UU$ is constructed iteratively, there must exist a chain of elements $T_1,\ldots,T_m=T$ in $\TT_\ell$ such that $T_1\in \MM_\ell$ and such that $T_j$ is the left neighbor of $T_{j+1}$ as well as $|T_{j+1}|/|T_j|>g_0$ for all $j=1,\ldots,m-1$.
It is clear that $T_1\in \widetilde \TT_\ell$, since refinement of $T_1$ would have triggered refinement of all $T_j$ and thus contradicts $T\in\widetilde \TT_\ell$. Therefore, $T_1\in \widetilde \MM_\ell$ and eventually also $T_m=T\in \widetilde \UU$. Hence $T$ is refined to obtain $\widetilde \TT_{\ell+1}$.

Case~3: Let $T\notin \widetilde \TT_\ell$. Then the induction hypothesis implies that $T$ was already refined and its successors are in $\widetilde \TT_\ell$. Thus, refining $T$ in $\TT_\ell$ will not create a mesh that is finer than $\widetilde \TT_{\ell+1}$.
This concludes the induction.

We have shown in particular that $\#\TT_\ell\leq \#\widetilde \TT_\ell$. Together with $\#\MM_\ell \geq \#\widetilde \MM_\ell$ for all $\ell$, the closure estimate for $\widetilde \TT_\ell$ implies the closure estimate for $\TT_\ell$.
This concludes the proof.
\end{proof}

The following two results show that mesh grading does not alter the best possible convergence rate. The arguments can already be found for the more complicated higher dimensional case in~\cite{stevensondemlow}.
\begin{redlemma}\label{lem:graded_refinement}
    Let $\TT_0$ be the initial mesh, and $\T$ be the collection of all meshes that can be obtained from $\TT_0$ by recurrent trisections ($k=2$) or bisections ($k>2$). Let $\T^{\rm grad}\subseteq \T$ be the collection of all meshes that additionally respect~\eqref{eq:graded}. Then, for every $\TT \in \T$, there exists a mesh $\TT^{\rm grad} \in \T^{\rm grad}$ such that $\TT^{\rm grad}$ is a refinement of $\TT$ (in the sense that each element $T\in\TT$ is either itself in $\TT^{\rm grad}$ or all of its descendants are in $\TT^{\rm grad}$) and it holds
    \begin{align*}
        \# \TT^{\rm grad} -\# \TT_0 \leq C \Big( \# \TT -\# \TT_0\Big).
    \end{align*}
    where $C>0$ is independent of $\TT$.
\end{redlemma}
\begin{proof}[Proof]
Let $\TT \in \T$. Starting from $\TT_0$, we construct a mesh $\TT^{\rm grad}$ iteratively as follows: For $\ell=0,1,\ldots,L$ set $\MM_\ell:=\{T \in \TT_\ell: \exists T' \in \TT, T' \subsetneq T\}$. Refine all elements $T \in \MM_\ell$ via Steps~5--9 in Algorithm~\ref{alg:timeadapt} to obtain $\TT_{\ell+1}$. The procedure terminates when $\MM_L=\emptyset$ and we obtain a mesh $\TT^{\rm grad}:= \TT_L \in \T^{\rm grad}$ that is a refinement of $\TT$. The mesh closure estimate \eqref{eq:closure} gives
    \begin{align*}
        \#\TT_L-\# \TT_0 \lesssim \sum_{\ell=0}^{L-1} \# \MM_\ell.
    \end{align*}
    To bound $\sum_{\ell=0}^{L-1} \# \MM_\ell$, we construct a similar sequence of meshes by recurrent trisections/bisections starting from $\TT_0=\widetilde\TT_0$, without any mesh grading. For $\ell=0,1,\ldots,\widetilde L$, set $\widetilde\MM_\ell:=\{T \in \widetilde\TT_\ell: \exists T' \in \TT, T' \subsetneq T\}$. Substitute all elements $T \in \widetilde\MM_\ell$, by its children in $\TT_T$ to obtain $\widetilde\TT_{\ell+1}$. This procedure terminates when $\widetilde\MM_{\widetilde L}=\emptyset$, and we obtain the mesh $\widetilde\TT_{\widetilde L}=\TT$. Note that, since the mesh grading procedure refines extra elements, we necessarily have $\sum_{\ell=0}^{L-1} \# \MM_\ell \leq \sum_{\ell=0}^{\widetilde L-1} \# \widetilde\MM_\ell$. Furthermore, since $\widetilde\TT_{\ell+1}$ is obtained via simple trisection/bisection of marked elements from $\widetilde \TT_\ell$, there holds
    \begin{align*}
    \# \widetilde\TT_{\ell+1} = \# \widetilde\TT_\ell+ \left\{\begin{array}{cc}1 & \text{for }k>2\\ 2 & \text{for }k=2\end{array}\right\} \# \widetilde\MM_\ell.
    \end{align*}
    This shows immediately
    \begin{align*}
        \# \TT = \# \TT_0+ \left\{\begin{array}{cc}1 & \text{for }k>2\\ 2 & \text{for }k=2\end{array}\right\}\sum_{\ell=0}^{\widetilde L-1} \# \widetilde\MM_\ell.
    \end{align*}
    The combination of all the steps yields
    \begin{align*}
        \#\TT^{\rm grad}-\# \TT_0 \lesssim \sum_{\ell=0}^{L-1} \# \MM_\ell \leq \sum_{\ell=0}^{\widetilde L-1} \# \widetilde\MM_\ell \leq \# \TT -\# \TT_0
    \end{align*}
    and concludes the proof.
\end{proof}
This result immediately gives that the best convergence rate $s$ in the class $\T$ is the same as in the class $\T^{\rm grad}$.
\begin{redlemma}\label{lem:graded_equiv}
    There holds for $s>0$ that
    \begin{align}
    \label{eq:approx1}
    \begin{split}
    C_{\rm approx}\leq 
       C_{\rm approx}^{\rm grad}:=\sup_{N \in \N} &\inf_{\substack{\TT \in \T^{\rm grad}\\\# \TT - \# \TT_0 \leq N}} \eta_\TT N^s\leq C_{\rm grad} (C_{\rm approx}+1)\\
       \end{split}
    \end{align}
    for a constant $C_{\rm grad}>0$ that depends only on $s$, $f$, $u$, and the constant from Lemma~\ref{lem:graded_refinement}.
\end{redlemma}
\begin{proof}[Proof]
Because $\T^{\rm grad} \subset \T$, $C_{\rm approx}\leq 
       C_{\rm approx}^{\rm grad}$ follows immediately. For the other estimate, we
    denote by $(\widetilde\TT_{\widetilde N})_{{\widetilde N} \in \N} \subset \T$ with $\#\widetilde\TT_{\widetilde N}-\# \TT_0\leq {\widetilde N}$ a supremizing sequence on the left-hand side of~\eqref{eq:approx1}. By Lemma \ref{lem:graded_refinement}, there exists a sequence of graded meshes $(\TT_{\widetilde N})_{\widetilde N \in \N} \subset \T^{\rm grad}$ with $\#\TT_{\widetilde N}-\# \TT_0\leq C\widetilde N$ such that $\TT_{\widetilde N}$ is a refinement of $\widetilde\TT_{\widetilde N}$. For every $N\in\N$ with $N \geq C$, there exists $ \widetilde N \in \N$ such that $C\widetilde N \leq N $ and $N \leq C(\widetilde N+1)$. This and Lemma~\ref{lem:qmon} show
    \begin{align*}
        \inf_{\substack{\TT \in \T^{\rm grad}\\\# \TT - \# \TT_0 \leq N}} \eta_\TT N^s \leq \eta_{\TT_{\widetilde N}}N^s 
        \leq \eta_{\TT_{\widetilde N}} C^s(\widetilde N+1)^s\leq  C_{\rm mon}\eta_{\widetilde\TT_{\widetilde N}} C^s(\widetilde N+1)^s\lesssim C_{\rm approx}.
        \end{align*}        
    For the finitely many remaining $N<C$, we use efficiency from Lemma~\ref{lem:reliability} and Corollary~\ref{cor:cea} to show $\eta_\TT\lesssim \norm{u-u_0}{\XX}+\norm{f}{H^1(0,\tend;V^\star)}$, which implies
    \begin{align*}
         \sup_{N \in \N} &\inf_{\substack{\TT \in \T^{\rm grad} \\ \# \TT - \# \TT_0 \leq N}} \eta_\TT N^s \lesssim C_{\rm approx} + 1
    \end{align*}
    and concludes the proof.
\end{proof}
\bibliographystyle{siamplain}
\bibliography{literature}
\end{document}